\documentclass[11pt,reqno, letterpaper]{amsart}
\usepackage{hyperref}
\usepackage{amsfonts,mathrsfs,bbm,latexsym,rawfonts,amsmath,amssymb,amsthm,epsfig}
\usepackage{fullpage, setspace, color}

\newtheorem{thm}{Theorem}[section]
\newtheorem{cor}[thm]{Corollary}
\newtheorem{rem}[thm]{Remark}
\newtheorem{lem}[thm]{Lemma}
\newtheorem{prop}[thm]{Proposition}
\newtheorem{defn}[thm]{Definition}

\numberwithin{equation}{section}

\newcommand{\al}{\alpha}

\newcommand{\De}{\Delta}
\newcommand{\ep}{\varepsilon}

\newcommand{\si}{\sigma}
\newcommand{\om}{\omega}
\newcommand{\Om}{\Omega}

\newcommand{\Ga}{\Gamma}

\renewcommand{\P}{\mathcal{P}}
\renewcommand{\L}{\mathcal{L}}

\newcommand{\U}{\mathbb{S}}

\newcommand{\Real}{\mathbb{R}}

\newcommand{\norm}[1]{\Vert#1\Vert}

\def\<{\left\langle} \def\>{\right\rangle}
\def\({\left(} \def\){\right)}
\newcommand{\n}{\nabla}
\newcommand{\p}{\partial}


\begin{document}
\title{Smooth solutions to the Schr\"odinger flow for maps from smooth bounded domains in Euclidean spaces into $\U^2$}
\thanks{*Corresponding Author}
\author{Bo Chen}
\address{School of Mathematics, South China University of Technology, Guangzhou,510640, China}
\email{cbmath@scut.edu.cn}
	
\author{Youde Wang*}
\address{1. School of Mathematics and Information Sciences, Guangzhou University;
2. Hua Loo-Keng Key Laboratory of Mathematics, Institute of Mathematics, AMSS, and School of
Mathematical Sciences, UCAS, Beijing 100190, China.}
\email{wyd@math.ac.cn}

\begin{abstract}
The results of this paper are twofold. First, we establish the local existence and uniqueness of very regular or smooth solutions to the initial-Neumann boundary value problem of the Schr\"{o}dinger flow for maps from a smooth bounded domain $\Om\subset \Real^m$ with $m=1,2,3$ into $\U^2$ in the scale of Sobolev spaces. In this part, we also provide a precise description of the compatibility conditions required at the boundary for the initial data. Second, we further prove that the local smooth solution obtained for the initial-Neumann boundary value problem of the 1-dimensional
Schr\"odinger flow can be extended to a global smooth one.
\end{abstract}

\maketitle
\section{Introduction}
In this paper, we are concerned with the existence and uniqueness for very regular or smooth solutions to the following initial-Neumann boundary value problem of the Schr\"odinger flow
\begin{equation*}
	\begin{cases}
		\p_tu =u\times\De u,\quad\quad&\text{(x,t)}\in\Om\times \Real^+,\\[1ex]
		\frac{\p u}{\p \nu}=0, &\text{(x,t)}\in\p\Om\times \Real^+,\\[1ex]
		u(x,0)=u_0: \Om\to \U^2,
	\end{cases}
\end{equation*}
where $\Om\subset\mathbb{R}^m$ with $m=1,2,3$ is a smooth bounded domain, and $u$ is a map from $\Om$ into a standard sphere $\U^2$. More precisely, we aim to identify the compatibility conditions and regularity requirements on the initial map $u_0:\Om\to\U^2$  that ensure the existence of a unique smooth (or sufficiently regular) solution to the above problem. The present paper can be regarded as the sequel of our recent paper \cite{CW1}, where we established the local well-posedness of strong solutions to the Schr\"odinger flow (also known as the Landau-Lifshitz equation) on a smooth bounded domain associated with natural boundary conditions.

A long-standing and challenging question in this area is whether the Schrödinger flow (Landau-Lifshitz equation) with natural boundary conditions admits regular or smooth solutions. To the best of our knowledge, there are very few well-posedness results for such smooth or regular solutions in the existing literature when $\dim(\Omega)\geq 2$, where $\dim(\Omega)$ denotes the dimension of the domain $\Omega$.

\subsection{Definitions and Background}
In physics, for a map $u$ from $\Om$ into a standard sphere $\U^2$, the Landau-Lifshtiz (LL) equation
\begin{equation}\label{LL-eq0}
\p_tu =-u\times\De u
\end{equation}
is a fundamental evolution equation for the ferromagnetic spin chain and was proposed on the phenomenological ground in studying the dispersive theory of magnetization of ferromagnets. It was first deduced by Landau and Lifshitz in \cite{LL}, and then proposed by Gilbert in \cite{G} with dissipation as the following form
\begin{equation}\label{eq-LL}
\p_tu=-\al u\times (u\times \De u)+\beta u\times\De u,	
\end{equation}
where $\beta$ is a real number and $\al \geq 0$ is called the Gilbert damping coefficient. Hence, the above equation \eqref{eq-LL} is also called Landau-Lifshitz-Gilbert (LLG) equaion. Here ``$\times$" denotes the cross product in $\Real^{3}$ and $\De$ is the Laplace operator in $\Real^{3}$.

Since the negative sign ``$-$'' in equation \eqref{LL-eq0} does not affect on our analysis and main results, for the sake of convenience, we only consider the classical Schr\"{o}dinger flow into $\U^2$
\[\p_tu =u\times\De u.\]

\medskip
Intrinsically, ``$u\times$" can be considered as a complex structure
\[J(u)=u\times: T_{u}\U^2\to T_{u}\U^2\]
on $\U^2$, which rotates vectors on the tangent space of $\U^2$ anticlockwise by an angle of $\frac{\pi}{2}$ degrees. Therefore, we can write the above equation as
\[\p_tu = J(u)(\De u + |\nabla u|^2u).\]

From the viewpoint of infinite dimensional symplectic geometry, Ding and Wang \cite{DW} proposed to consider the Schr\"odinger flows for maps from a Riemannian manifold into a symplectic manifold, which can be regarded as an extension of LL equation \eqref{LL-eq0} and was also independently introduced by Terng and Uhlenbeck in \cite{Uh1}. Namely, suppose $(M,g)$ is a Riemannian manifold and $(N,J,\om)$ is a symplectic manifold, the Schr\"odinger flow is a time-dependent map $u:M\times \Real^+\to N\hookrightarrow \Real^{K}$ satisfying
\[\p_t u=J(u)\tau(u).\]
Here $\tau(u)$ is the tension field of $u$ which has the extrinsic form
$$\tau(u)=\De_g u+A(u)(\n u,\n u),$$
where $A(u)(\cdot, \cdot)$ is the second fundamental form of $N$ in $\Real^{K}$.

\medskip
The PDE aspects of the Schr\"odinger flow containing the existence, uniqueness and regularities of various kinds of solutions, have been intensively studied in the last two decades. Next, we briefly recall a few of results that are closely related to our work in the present paper.

\medskip
In 1986, P.L. Sulem, C. Sulem and C. Bardos in \cite{SSB} proved the existence of global weak solutions and local regular solutions to the Schr\"{o}dinger flow for maps from $\Real^n$ into $\U^2$, by employing difference method. In 1998, Y.D. Wang \cite{W} adopted a geometric approximation method (i.e. the complex structure approximation method) to obtain the global existence of weak solutions to the Schr\"odinger flow for maps from a closed Riemannian manifold or a bounded domain in $\Real^n$ into $\U^2$. Later, A. Nahmod, J. Shatah, L. Vega and C.C. Zeng \cite{NSVZ} established the existence of global weak solutions to the Schr\"odinger flow from $\Real^2$ into $\mathbb{H}^2$. For recent development of weak solutions to a class of generalized Schr\"odinger flows and related equations, we refer to \cite{CW0,JW1,JW2} and references therein for various results.

\medskip
The local regular solution to the Schr\"odinger flow from a closed Riemanian manifold or $\Real^n$ into a K\"ahler manifold was established by Ding and the second named author of this paper in \cite{DW,DW1} by employing a parabolic geometric approximation equation and estimating some intrinsic geometric energy picked suitably. Furthermore, they also obtained the persistence of regularity results, in that the solution always stays as regular as the initial data (as measured in Sobolev norms), provided that one is within the time of existence guaranteed by the local existence theorem.

For low-regularity initial data, the initial value problem for Schr\"odinger flow from an Euclidean space into $\mathbb{S}^2$ has been studied indirectly using the ``modified Schr\"odinger map equations" and certain enhanced energy methods, for instance, A.R. Nahmod, A. Stefanov and K. K. Uhlenbeck \cite{NSU} have ever used the standard technique of Picard iteration in some suitable function spaces of the Schr\"odinger equation to obtain a near-optimal (but conditional) local well-posedness result for the Schr\"odinger map flow equation from two dimensions into the sphere $X = \mathbb{S}^2$ or hyperbolic space $X = \mathbb{H}^2$. Moreover, they also proved the persistence of regularity results, in that the solution always stays as regular as the initial data, provided that one is within the time interval of existence guaranteed by the local existence theorem.

For one dimensional global existence for Schr\"odinger flow from $\mathbb{S}^1$ or $\mathbb{R}^1$ into a K\"ahler manifold, we refer to \cite{PWW, RRS, ZGT} and references therein. The global well-posedness result for the Schr\"odinger flow from $\mathbb{R}^n$ (with $n\geq 2$) into $\U^2$ with small initial data was well researched by Ionescu, Kenig, Bejanaru and et al., we refer to \cite{B1, BIK, BIKT, IK} for more details. Especially, in \cite{BIKT} the global well-posedness result for the Schr\"odinger flow for small data in the critical Sobolev spaces in dimensions $n\geq2$ was addressed. Recently, in \cite{L1, L2} Z. Li proved that the Schr\"odinger flow from $\mathbb{R}^n$ with $n\geq 2$ to compact K\"ahler manifold with small initial data in critical Sobolev spaces is also global well-posed.

\medskip
On the contrary, F. Merle, P. Rapha\"el and I. Rodnianski \cite{MRR} considered the energy critical Schr\"odinger flow problem with the 2-sphere target for equivariant initial data of homotopy index $k=1$. They showed the existence of a codimension one set of smooth well localized initial data arbitrarily close to the ground state harmonic map in the energy critical norm, which generates finite time blowup solutions, and gave a sharp description of the corresponding singularity formation which occurs by concentration of a universal bubble of energy. One also found some self-similar solutions to Schr\"odinger flow from $\mathbb{C}^n$ into $\mathbb{C}P^n$ with local bounded energy which blow up at finite time, for more details we refer to \cite{DTZ, GSZ,NSVZ}.

\medskip
As for some travelling wave solutions with vortex structures, F. Lin and J. Wei \cite{LW} employed perturbation method to consider such solutions for the Schr\"odinger map flow equation with easy-axis and proved the existence of smooth travelling waves with bounded energy if the velocity of travelling wave is small enough. Moreover, they showed the travelling wave solution has exactly two vortices. Later, J. Wei and J. Yang \cite{WY} considered the same Schr\"odinger map flow equation as in \cite{LW}, i.e. the Landau-Lifshitz equation describing the planar ferromagnets. They constructed a travelling wave solution possessing vortex helix structures for this equation. Using the perturbation approach, they give a complete characterization of the asymptotic behaviour of the solution.

\medskip
It should also be pointed out that Banica and Vega in \cite{Ba, Ba1} studied the stability properties of self-similar solutions of the geometric (Da Rios) flow
\[\gamma_t = \gamma_x \times \gamma_{xx},\]
which was proposed by Da Rios. Here $\gamma = \gamma(t, x)\in\mathbb{R}^3$, $x$ denotes the arclength parameter and $t$ the time variable. It is well-known that Da Rios flow is directly related to Schr\"oding flow. Based on the Hasimoto transform, the problem is reduced to the long-time asymptotics of the cubic nonlinear Schr\"odinger equation with time-dependent coefficients in one space dimension, where the Cauchy data are supposed to be small regular perturbations of constant given at $t= 1$.

They also made a connection between a famous analytical object introduced in the 1860's by Riemann, as well as some variants of it, and the Da Rios flow (the binormal curvature flow) in \cite{Ba4}. As a consequence, this analytical object has a non-obvious nonlinear geometric interpretation. Moreover, they proved the existence of solutions of the binormal curvature flow with smooth trajectories that are as close as desired to curves with a multifractal behavior, and showed that this behavior falls within the multifractal formalism of Frisch and Parisi \cite{FP}, which is conjectured to govern turbulent fluids. For more details we refer to \cite{Ba, Ba1, Ba2, Ba3, Ba4}.

\medskip
On the other hand, the LLG equation with initial-Neumann boundary conditions has consistently attracted interest from both physicists and mathematicians:
\begin{equation*}
	\begin{cases}
		\p_tu =-\al u\times(u\times\De u)+\beta u\times\De u,\quad\quad&\text{(x,t)}\in\Om\times \Real^+,\\[1ex]
		\frac{\p u}{\p \nu}=0, &\text{(x,t)}\in\p\Om\times \Real^+,\\[1ex]
		u(x,0)=u_0: \Om\to \U^2,
	\end{cases}
\end{equation*}
along with related problems stemming from certain systems connected to the LLG equations (see \cite{CF,S-R}). Here $\nu$ denotes the outer normal vector of $\p\Om$ and $u_{0}$ is the initial data.

Next, we retrospect some of the works related to local regular solutions of the initial-Neumann boundary value problem to LLG equations (i.e. equation \eqref{eq-LL}) with $\al>0$. In the case where the base space is a bounded domain $\Om\subset \Real^{3}$, Carbou and Fabrie proved the local existence and uniqueness of regular solutions to a dissipative LL equation coupled with Maxwell equations in micromagnetism theory in \cite{CF}. Recently, the local existence of very regular solution to LLG equation with $\al>0$ was addressed by applying the delicate Galerkin approximation method and adding compatibility initial-boundary condition in \cite{CJ}. Inspired by the method used in \cite{CJ}, we obtained the local in time very regular solution to LLG equation with spin-polarized transport in \cite{CW}.

\subsection{Main results on local regular solutions to the Schr\"odinger flow}
For the most challenging case where $\al=0$, there is a few results in the literature about the well-posedness of the initial-Neumann boundary value problem of the Schr\"odinger flow
\begin{equation}\label{S-eq}
	\begin{cases}
		\p_tu =u\times\De u,\quad\quad&\text{(x,t)}\in\Om\times \Real^+,\\[1ex]
		\frac{\p u}{\p \nu}=0, &\text{(x,t)}\in\p\Om\times \Real^+,\\[1ex]
		u(x,0)=u_0: \Om\to \U^2;
	\end{cases}
\end{equation}
except for the authors of the present paper obtained the existence and uniqueness of local strong solutions to \eqref{S-eq} by assuming that $u_0\in H^3(\Om)$ with $\frac{\p u_0}{\p \nu}|_{\p \Om}=0$ in the previous work \cite{CW1}. In this paper, we continue to investigate the local existence of very regular solutions to the equation, provided some necessary compatibility conditions of the initial data. Our main conclusions are presented as follows.

\begin{thm}\label{thm1}
Let $\Om$ be a smooth bounded domain in $\Real^3$. Suppose that $u_0\in H^5(\Om,\U^2)$, which satisfies  the $1$-order compatibility condition defined in \eqref{com-cond2}, i.e. $\frac{\p u_0}{\p \nu}|_{\p \Om}=0$ and $\tilde{\n}_\nu \tau(u_0)|_{\p\Om}=0$, where $\tilde{\n}$ is the pull-back connection on ${u_0}^*T\U^2$. Then there exists a positive time $T_1$ depending only on $\norm{u_0}_{H^5(\Om)}$ such that the initial-Neumann boundary value problem \eqref{S-eq} admits a unique local in time regular solution $u$, which satisfies
\[\p^i_tu\in L^\infty([0,T_1], H^{5-2i}(\Om))\]
for $i=0,1,2$.
\end{thm}
In general, let $u_0\in H^{2k+2}(\Om, \U^2)$ with $k\geq 1$. We can show the existence of very regular solutions to \eqref{S-eq} under adding $k$-order compatibility condition $CC(k)$ for $u_0$ (one can also see \eqref{s-com-cond} in Definition \ref{stronger-com-cond}):
\begin{itemize}
\item[$\blacklozenge$]  \emph{For any $1\leq j\leq 2k$, there holds
\begin{equation*}
		\frac{\p}{\p \nu}\p^ju_0|_{\p\Om}=0,
\end{equation*}
where $\p^ju_0=\(\frac{\p^ju_0}{\p x^{i_1}\cdots\p x^{i_j}}\)$ are all the $j$-th partial derivatives of $u_0$.}	
\end{itemize}

\begin{thm}\label{thm2}
Suppose that $u_0\in H^{2k+1}(\Om,\U^2)$ with $k\geq 2$, which satisfies the $(k-1)$-order compatibility condition $CC(k-1)$. Let $u$ 
and $T_1>0$ be given in Theorem \ref{thm1}. Then, for any $0\leq i\leq k$ we have
\[\p^i_t u\in L^\infty([0,T_1], H^{2k+1-2i}(\Om)).\]

Additionally, if $u_0\in C^\infty(\bar{\Om})$, which satisfies the compatibility condition $CC(k)$ for any $k\geq 0$, we also have
\[u\in C^\infty(\bar{\Om}\times [0,T_1]).\]
\end{thm}

\begin{rem}\label{higher-dim-case}
\begin{itemize}
	\item[$(1)$] We should point out that the conclusions of Theorems \ref{thm1} and \ref{thm2} continue to hold in the following cases: when $\Om$ is a smooth bounded domain in $\Real^m$ for $m = 1$ or $2$, or when $\Om$ is a compact manifold with smooth boundary of dimension at most three.
	\item[$(2)$]  It seems that the results stated in Theorems \ref{thm1} and \ref{thm2} can be extended to the case where the target manifold of the Schr\"odinger flow is a compact K\"ahler manifold beyond the sphere $\U^2$. But we need to use some different arguments and techniques from here and to overcome some new essential difficulties.
	\item[$(3)$] It seems that our current arguments in the proof of Theorems \ref{thm1} and \ref{thm2} may not valid when the dimension of the domain $\Om$ is larger than 3. This raises an open question: Can the existence of regular solutions to problem \eqref{S-eq} be established in the case where the dimension of $\Om$ exceeds 3?
\end{itemize}
\end{rem}

In particular, for the one-dimensional case, we can establish the following existence result for the Schr\"odinger flow. Let $I=[0,1]$ for simplicity, the 1-dimensional Schr\"odinger flow satisfies

\begin{equation}\label{eq-ISMF0}
	\begin{cases}
		\p_tu=u\times \p^2_xu,&\text{(x,t)}\in(0,1)\times \Real^+,\\[1ex]
		\p_xu(0,t)=0,\,\p_xu(1,t)=0, &t\in\Real^+,\\[1ex]
		u(x,0)=u_0: I\to \U^2,
	\end{cases}
\end{equation}
where we set $\Om=(0,1)\subset\Real^1$ with coordinate $\{x\}$, $u$ is a time-dependent map from $(0,1)$ into a standard sphere $\U^2$.
	
Let $u_0\in H^{2k+2}(I, \U^2)$ with $k\geq 1$. We say $u_0$ satisfies the necessary $k$-order compatibility condition $\widetilde{CC}(k)$ (see Proposition \ref{intrc-cd2}), if
\begin{itemize}
\item[$\blacklozenge$] \emph{For any $0\leq j\leq k$, there holds	
\begin{equation*}
\tilde{\n}^{2j+1}_xu_0|_{\p I}=0.
\end{equation*}}
\end{itemize}
It is worth noting that this compatibility condition $\widetilde{CC}(k)$ is weaker than the condition $CC(k)$.
\begin{thm}\label{thm2'}
Suppose that $u_0\in H^{2k+1} (I,\U^2)$ with $k\geq 2$, which satisfies the $(k-1)$-order compatibility condition $\widetilde{CC}(k-1)$. Then there exists a positive time $T_1$ depending only on $\norm{u_0}_{H^5(I)}$ such that the initial-Neumann boundary value problem \eqref{eq-ISMF0} admits a unique local regular solution $u$ on $[0,T_1]$ such that for any $0\leq i\leq k$ we have
\[\p^i_t u\in L^\infty([0,T_1], H^{2k+1-2i}(I)).\]

Additionally, if $u_0\in C^\infty([0,1])$, which satisfies the compatibility condition $\widetilde{CC}(k)$ for any $k\geq 0$, we also have
\[u\in C^\infty(I\times [0,T_1]).\]
\end{thm}

\medskip
We adopt a similar parabolic perturbation approximation of \eqref{S-eq} and use geometric energy method with that in \cite{DW} to get very regular solutions to \eqref{S-eq}. Indeed, we will use the local very regular solutions to the parabolic perturbed equation of \eqref{S-eq}
\begin{equation}\label{pra-eq0}
	\begin{cases}
		\p_tu =\ep \tau(u)+u\times \De u,\quad\quad&\text{(x,t)}\in\Om\times \Real^+,\\[1ex]
		\frac{\p u}{\p \nu}=0, &\text{(x,t)}\in\p\Om\times \Real^+,\\[1ex]
		u(x,0)=u_0: \Om\to \U^2,
	\end{cases}
\end{equation}
with $0<\ep<1$ to approximate a regular solution to the problem \eqref{S-eq}, where
\[\tau(u_\ep)=\De u_\ep+|\n u_\ep|^2u_\ep=-u_\ep\times(u_\ep\times \De u_\ep),\]
since $|u_\ep|=1$. The key point is to establish some suitable uniform high order energy estimates of $u_\ep$ with respect to $\ep$.

For each $\ep>0$, recall that the local existence of very regular solution to \eqref{pra-eq0} has been established in \cite{CJ}(also see \cite{CW}), which can be formulated as the following theorem.
\begin{thm}
Suppose that $u_0\in H^{2k+1}(\Om, \U^2)$ with $k\geq 1$, and satisfies the $(k-1)$-order compatibility condition \eqref{com-cond}. Then there exists a positive time $T_\ep$ (depending only on $\ep$ and $\norm{u_0}_{H^2(\Om)}$) such that the problem \eqref{pra-eq0} admits a unique local solution $u_\ep$, which satisfies
\[\p^i_tu_\ep\in L^\infty([0,T], H^{2k+1-2i}(\Om))\cap L^2([0,T], H^{2k+2-2i}(\Om))\]
for any $0<T<T_\ep$  and $0 \leq i\leq k$.
\end{thm}

Next, we outline the strategy and main ideas for addressing the above perturbed problem. Since the initial data must meet the necessary compatibility condition defined by \eqref{com-cond} in the above theorem, the first difficulty we encounter is: how to find an initial data $u_0$ such that it satisfies the compatibility condition on boundary which is independent of $\ep\in (0,1)$? The answer lies in the following approach.

By applying the intrinsic geometric structures of the equation $\p_tu_\ep=\ep \tau(u_\ep)+J(u_\ep)\tau(u_\ep)$:
\begin{itemize}
	\item $\ep\tau(u_\ep)$ and $J(u_\ep)\tau(u_\ep)$ is orthogonal,
	\item The complex structure $J$ is integrable, i.e. $\n J=0$,
\end{itemize}
we are able to provide an intrinsic description of compatibility condition of the initial data elucidated in Proposition \ref{intrc-cd1} (or Proposition \ref{intrc-cd2}), which implies the 1-order compatibility condition as well as the any $k$-order compatibility conditions for 1-dimensional case for equation \eqref{S-eq} are the same as that for its parabolic perturbed equation \eqref{pra-eq0}. Moreover, under the stronger but natural restrictions of $u_0$ (i.e \eqref{s-com-cond}) given in Definition \ref{stronger-com-cond}, we can also show these two equations in general dimensional case are of the same $k$-order compatibility conditions with $k>1$. Further details are provided in Section~\ref{s: com-cond}.

\medskip
Secondly, we need to derive uniform high-order energy estimates of approximate solution $u_\ep$, which are independent of $\ep\in(0,1)$. However, since the space of the test functions associated to the initial-Neumann boundary problem \eqref{pra-eq0} (i.e. those functions vanishing the boundary terms when integration by parts are applied) is much smaller than that in \cite{DW}, there are two essential difficulties to overcome in this step:
\begin{itemize}
\item[$(1)$] One is how to find test functions associated to the initial-Neumann boundary problem \eqref{pra-eq0}?
\item[$(2)$] The other is how to get uniform high order energy estimates avoiding the loss of derivatives by using these test functions in $(1)$?
\end{itemize}

We will make full use of the geometric structures of the Schr\"odinger flow \eqref{S-eq} to overcome these two issues. Because \eqref{S-eq} has the following extrinsic geometric structures: 1. $J=u\times$ is antisymmetric; 2. $(\times, \Real^3)$ is a Lie algebra; 3. The fact $|u|=1$ implies $u\in T^\perp_u\U^2$, we can choose suitable extrinsic geometric energy which can control the energy for us to estimate.

On the other hand, we need to use the following fact: for any $k\in \mathbb{N}$ and any $u\in H^{k+2}(\Om)$ with $\frac{\p u}{\p \nu}|_{\p\Om}=0$, there holds the following inequality
\[\norm{u}_{H^{k+2}}(\Om)\leq C(\norm{\De u}_{H^k(\Om)}+\norm{u}_{L^2(\Om)}),\]
which means that $\norm{\De u}_{H^k(\Om)}+\norm{u}_{L^2(\Om)}$ is an equivalent Sobolev norm to $\norm{u}_{H^{k+2}}$.
For a regular solution $u_\ep$ to \eqref{pra-eq0}, setting $w_k=\p^k_t u_\ep$, the facts
\[\frac{\p w_k}{\p \nu}|_{\p\Om\times[0,T_\ep)}=0,\]
for each $k\in \mathbb{N}$ tell us that
\begin{itemize}
\item[$(i)$] There hold true the equivalent Sobolev estimates
\begin{align*}
\norm{w_k}_{H^2}\leq& C(\norm{\De w_k}_{L^2}+\norm{w_k}_{L^2}),\\
\norm{w_k}_{H^3}\leq& C(\norm{\De w_k}_{H^1}+\norm{w_k}_{L^2}).
\end{align*}
By writing the equation \eqref{pra-eq0} as the following equivalent form
\[\De u_\ep=\frac{1}{1+\ep^2}(\ep \p_t u_\ep-u_\ep\times \p_t u_\ep)-|\n u_\ep|^2 u_\ep,\]
from the above inequalities one can infer the key estimates of equivalent Sobolev norms in Lemma \ref{v-w} and Lemma \ref{w_{k-1}-w_k}.

\medskip
\item[$(ii)$] $w_k$ and $\De w_k$ can be chosen as the suitable test functions matching the Neumann boundary conditions.
\end{itemize}

\medskip
The above two observations imply that we should consider the equation of $w_k=\p^k_tu_\ep$ with the compatibility condition of initial data \eqref{com-cond}, the uniform higher order geometric energy estimates of $u_\ep$ can be obtained after we showing the key estimates for equivalent Sobolev norms of $w_k$ stated in $(i)$. More precisely, we will use the following simple process $\P$ to explain the strategy of improving the order of energy estimates.
\begin{itemize}
\item[$(1)$] Assume that $u_0\in H^3(\Om,\U^2)$ and satisfies the $0$-order compatibility condition, i.e. $\frac{\p u_0}{\p \nu}|_{\p \Om}=0$. By considering the equation satisfied by $w_1=\p_t u_\ep$ and applying the key $H^3$-equivalent norms of $u_\ep$ established in \cite{CW1}:
\[\norm{u_\ep}^2_{H^3}\leq C(\norm{\De u_\ep}^2_{L^2}+\norm{w_1}^2_{H^1}+1)^3,\]
we can give a uniform $H^3$-bound of $u_\ep$ on some uniform time interval $[0,T_0]$.

\medskip
\item[$(2)$] Assume $u_0\in H^5(\Om,\U^2)$ and satisfies the $1$-order compatibility condition, i.e. $\frac{\p u_0}{\p \nu}|_{\p \Om}=0$ and $\tilde{\n}_\nu \tau(u_0)|_{\p\Om}=0$. By using the equation satisfied by $w_2=\p^2_t u_\ep$ and applying the estimates obtained in $(1)$ and the key equivalent $H^3$-norm of $w_1$ in Lemma \ref{v-w}:
\begin{align*}
	\norm{w_1}^2_{H^2(\Om)}\leq &C(\norm{u_\ep}^4_{H^3}+1)\norm{w_1}^2_{H^1}+C\norm{w_2}^2_{L^2},\\
	\norm{w_1}^2_{H^3(\Om)}\leq &C(\norm{u_\ep}^2_{H^3},\norm{w_1}^2_{H^1})(\norm{w_2}^2_{H^1}+1),
\end{align*}
we can show a uniform $H^1$-estimate of $w_2$ on $[0,T_1]$ for some $0<T_1\leq T_0$. This implies a uniform $H^5$-bound on $u_\ep$ by using equation \eqref{pra-eq0}.

\medskip
\item[$(3)$] Letting $\ep \to 0$, we get a $H^5$-regular solution to \eqref{S-eq}. On the other hand, the uniqueness of such solution has been established in \cite{CW1}.
\end{itemize}
This completes the outline of the proof of Theorem \ref{thm1}.

\medskip
To get the higher regularity of the solution $u$ obtained in Theorem \ref{thm1}, we need to impose a stronger higher order compatibility condition as in Definition \ref{stronger-com-cond} (or a necessary higher order compatibility condition in \eqref{intrc-cd2} for 1-dimensional case). Then we can prove Theorem \ref{thm2} (or Theorem \ref{thm2'}) by using the method of induction, repeating the above process $\P$ for the higher order derivatives of $u_\ep$ in direction of time $t$ and applying the corresponding key equivalent norms established in Lemma \ref{w_{k-1}-w_k}.

\subsection{Global smooth solutions to 1-dimensional Schr\"odinger flow}
Once we get the local existence of smooth solution to the initial-Neumann boundary problem of the Schr\"odinger flow on bounded domains, another natural question is that whether the local solutions are globally well-posed? For 1-dimensional Schr\"odinger flow \eqref{eq-ISMF0} we can get a positive answer to this question. The precise result is stated in the following theorem.

\begin{thm}\label{thm3}
Suppose that $u_0\in C^\infty (I,\U^2)$, which satisfies the $k$-order compatibility condition $\widetilde{CC}(k)$ for any $k\geq 0$. Then the initial-Neumann boundary value problem \eqref{eq-ISMF0} admits a unique global smooth solution $u$ on $[0,\infty)$.
\end{thm}

The proof of this theorem is different from that of Theorems \ref{thm1} and \ref{thm2}. The key point is that we find the local smooth solution $u$ to \eqref{eq-ISMF0} satisfies a conversation law:
\begin{equation}\label{Key-formu}
\frac{\p}{\p t}\(\int_{I}|\p_tu|^2dx-\frac{1}{4}\int_{I}|\p_xu|^4dx\)=0,
\end{equation}
which was proved in \cite{SSB} for 1-dimensional Schr\"odinger flow from $\mathbb{R}^1$ into $\U^2$ (also see \cite{DW1}), and then was generalized to Hermitian locally symmetric spaces in \cite{DWW, PWW}. Then by applying the conversation law of energy:
\[\int_{I}|\p_xu|^2dx(t)=\int_{I}|\p_xu_{0}|^2dx\]
and the Sobolev interpolation inequality on $I$:
\[\int_{I}|\p_xu|^4\leq C\norm{\p_xu}_{H^1}\norm{\p_xu}^3_{L^2},\]
this implies
\[\norm{u}^2_{H^2}(t)\leq C(\norm{u_{0}}^2_{H^1}+1)^3+\norm{\tau(u_0)}^2_{L^2},\]
for any existence time $t$.

With this uniform $H^2$-estimate of $u$ at hand, by considering the equation of $\p^k_tu$ with $k\geq 1$ and estimating the corresponding high order equivalent energy $\norm{\p^k_t u}^2_{H^1}$, we can apply the method of induction analogous to that in the proof of Theorems \ref{thm1} and \ref{thm2} to get uniform bounds:
\[\sup_{0<t<T}\norm{\p^j_t\p^s_xu}^2_{L^2}\leq C(T),\]
for any $j,s\geq 0$, where $C(T)$ satisfies $C(T)<\infty$ if $T<\infty$. Therefore, Theorem \ref{thm3} follows from the above uniform estimates of $u$.

\medskip
The rest of our paper is organized as follows. In Section \ref{s: pre}, we introduce some basic notations on Sobolev space and some preliminary lemmas. The compatibility conditions will be given and described intrinsically in Section \ref{s: com-cond}. In Section \ref{s: H^5}, we prove Theorem \ref{thm1}. Next, Theorems \ref{thm2} and \ref{thm2'} are given in Section \ref{s: hig-reg}. Finally, we prove Theorem \ref{thm3} in Section \ref{s: global-exist}.

\section{Preliminary}\label{s: pre}
In this section, we begin with introducing some notions and notations on Sobolev spaces that will be used in the subsequent context of this paper. Let $\Om$ be a smooth bounded domain in $\Real^m$, N be an isometrically embedded submanifold of $\Real^K$. In many cases throughout the paper, we will take the standard sphere $\U^2$ in $\Real^3$ as our choice for $N$. Let $u:\Om\to N\hookrightarrow\Real^K$ be a map. We set
$$H^{k}(\Om,N)=\{u\in H^k(\Om)=W^{k,2}(\Om,\Real^K):u(x)\in N\,\,\text{for a.e. x}\in \Om\}.$$
For simplicity, we also denote $H^k(\Om)=W^{k,2}(\Om,\Real^K)$.

Moreover, let $(B,\norm{.}_B)$ be a Banach space and $f:[0,T]\to B$ be a map. For any $p>0$ and $T>0$, recall that
\[\norm{f}_{L^p([0,T], B)}:=\(\int_{0}^{T}\norm{f}^p_{B}dt\)^{\frac{1}{p}},\]
and
\[L^p([0,T],B):=\{f:[0,T]\to B:\norm{f}_{L^p([0,T],B)}<\infty\}.\]
In particular, we denote
\[L^{p}([0,T],H^{k}(\Om,N))=\{u\in L^{p}([0,T],H^{k}(\Om)):u(x,t)\in N\,\,\text{for a.e. (x,t)}\in \Om\times[0,T]\},\]
where $k,\,l\in \mathbb{N}$  and $p\geq 1$.

\subsection{Some preliminary lemmas}\
For later convenience, we need to recall some important preliminary lemmas. The $L^2$ theory of Laplace operator with Neumann boundary condition implies the following Lemma on equivalent Sobolev norms, for the details we refer to \cite{Weh}.

\begin{lem}\label{eq-norm}
Let $\Om$ be a bounded smooth domain in $\Real^{m}$ and $k\in\mathbb{N}$. There exists a constant $C_{k,m}$ such that, for all $u\in H^{k+2}(\Om)$ with $\frac{\p u}{\p \nu}|_{\p\Om}=0$,
\begin{equation}\label{eq-n}
\norm{u}_{H^{2+k}(\Om)}\leq C_{k,m}(\norm{u}_{L^{2}(\Om)}+\norm{\De u}_{H^{k}(\Om)}).
\end{equation}
Here, for simplicity we denote $H^0(\Om):=L^2(\Om)$.
\end{lem}
In particular, the above lemma implies that we can define the $H^{k+2}$-norm of $u$ as follows
$$\norm{u}_{H^{k+2}(\Om)}:=\norm{u}_{L^2(\Om)}+\norm{\De u}_{H^k(\Om)}.$$

\medskip
In order to show the uniform estimates and the convergence of solutions to the approximate equation constructed in coming sections, we also need to use the Gronwall inequality and the classical compactness results in \cite{BF, Sim}.

\begin{lem}\label{Gron-inq}
Let $f: \Real^+\to \Real^+$ be a nondecreasing continuous function such that $f>0$ on $(0,\infty)$ and $\int_{1}^{\infty}\frac{1}{f}dx<\infty$. Let $y$ be a continuous function which is nonnegative on $\Real^+$ and let $g$ be a nonnegative function in $L^{1}_{loc}(\Real^+)$. We assume that there exists a $y_0>0$ such that for all $t\geq0$, we have the inequality
	\[y(t)\leq y_0+\int_{0}^{t}g(s)ds+\int_{0}^{t}f(y(s))ds.\]
	Then, there exists a positive number $T^*$ depending only on $y_0$ and $f$, such that for all $T<T^*$, there holds
	\[\sup_{0\leq t\leq T}y(t)\leq C(T,y_0),\]
	for some constant $C(T,y_0)$.
\end{lem}

To take an analogous argument to the proof of Lemma \ref{Gron-inq} in \cite{BF}, one can easily show the following result.
\begin{cor}\label{ode}
Let $f: \Real^+\to \Real^+$ be a positive locally Lipschitz function, which is nondecreasing. Let $z:[0,T^*)\to \Real$ be the maximal solution of the Cauchy problem:
\begin{equation*}
\begin{cases}
z^\prime=f(z),\\[1ex]
z(0)=z_0.	
\end{cases}
\end{equation*}
Let $y:\Real^+\to \Real$ be a $W^{1,1}$ function such that
\begin{equation*}
\begin{cases}
y^\prime\leq f(y),\\[1ex]
y(0)=y_0\leq z_0.	
\end{cases}
\end{equation*}
Then, for any $0<T<T^*$, we have
$$y(t)\leq z(T),\quad\text{t}\in [0,T].$$
\end{cor}
\begin{proof}
Let
$$w(t)=y_0+\int_{0}^{t}f(y(s))dx.$$
It is easy to see that $w$ is a nondecreasing $C^1$ function, which satisfies
\begin{equation*}
\begin{cases}
w^\prime=f(y(t))\leq f(w(t)),\\[1ex]
w(0)=y_0\leq z_0.	
\end{cases}
\end{equation*}
Here we have used the fact that $f$ is positive and nondecreasing. Then, the classical ODE comparison theorem tells us that
\[w(t)\leq z(t)\]
for any $t\in [0,T^*)$. Therefore, we get the desired result since $y(t)\leq w(t)$.
\end{proof}

\begin{lem}[Aubin-Lions-Simon compactness Lemma, see Theorem II.5.16 in \cite{BF} or \cite{Sim}]\label{A-S}
Let $X\subset B\subset Y$ be Banach spaces. Suppose that the embedding $B\hookrightarrow Y$ is continuous and that the embedding $X\hookrightarrow B$ is compact. Let $1\leq p,q,r\leq \infty$. For $T>0$, we define
	\[E_{p,r}=\{f\in L^{p}((0,T), X), \frac{d f}{dt}\in L^{r}((0,T), Y)\}.\]
Then, the following properties hold true
\begin{itemize}
\item[$(1)$] If $p< \infty$ and $p<q$, the embedding $E_{p,r}\cap L^q((0,T), B)$ in $L^s((0,T), B)$ is compact for all $1\leq s<q$.
\item[$(2)$] If $p=\infty$ and $r>1$, the embedding of $E_{p,r}$ in $C^0([0,T], B)$ is compact.
\end{itemize}
\end{lem}
\begin{lem}[Theorem II.5.14 in \cite{BF}]\label{C^0-em}
	Let $k\in \mathbb{N}$, then the space
	\[E_{2,2}=\{f\in L^{2}((0,T),H^{k+2}(\Om)),\frac{\p f}{\p t}\in L^{2}((0,T), H^k(\Om))\}\]
	is continuously embedded in $C^0([0,T], H^{k+1}(\Om))$.
\end{lem}

\medskip
\section{Compatibility conditions }\label{s: com-cond}
In this section, we introduce the compatibility conditions on the initial data, which make the Schr\"odinger flow \eqref{S-eq} admits a regular or smooth solution. The main purpose is to find what kind of initial data can guarantee that equation \eqref{S-eq} and its parabolic perturbed equation \eqref{pra-eq0} have the same compatibility conditions. 

This section is structured as follows. In Subsection \ref{def-cc}, we define the compatibility conditions on the initial data $u_0$ and provide an equivalent intrinsic characterization. These conditions are necessary for the existence of regular solutions to problem \eqref{S-eq} and its parabolic perturbation. In Subsection \ref{uni-cc}, using the intrinsic characterization of the compatibility conditions, we show that the first-order compatibility condition for the Schr\"odinger flow, as well as any $k$-th order compatibility conditions in the one-dimensional case, coincides with that of its parabolic perturbed equation. Furthermore, under the stronger yet natural assumptions on $u_0$ specified in Definition \ref{stronger-com-cond}, we demonstrate that both equations also share the same $k$-th order compatibility conditions for $k>1$ in the general dimensional setting. Finally, in Subsection \ref{another-cc}, we introduce another compatibility conditions, which will be used to eliminate boundary terms during energy estimates in the following sections.

\subsection{Compatibility conditions of the initial data}\label{def-cc}
In general, let $(N, J, \omega)$ be a K\"ahler manifold, where $\om$ is the K\"ahler form and $J: TN\to TN$ with $J^2=-id$ is the complex structure, and we always assume that $N$ is an embedded submanifold of $\Real^K$ with second fundamental form $A(\cdot, \cdot)$. Let $\Om$ be a bounded smooth domain in $\Real^3$, equipped with Euclidean coordinates $\{x^1,x^2,x^3\}$.

For the sake of convenience, we assume $u$ is a smooth solution to the initial-Neumann boundary value problem of the perturbed equation of the Schr\"odinger flow on $\bar{\Om}\times [0,T]$ for some $T>0$:
\begin{equation}\label{eq-ASMF}
\begin{cases}
\tilde{\n}_t u = \ep\tau(u)+J(u) \tau(u),\quad\quad&\text{(x,t)}\in \Om\times [0,T],\\[1ex]
\frac{\p u}{\p \nu}=0, &\text{(x,t)}\in\p \Om\times [0,T],\\[1ex]
u(x,0)=u_0: \Om\to N\hookrightarrow\Real^{K},
\end{cases}
\end{equation}
for $\ep\in [0,1]$. Here $\tau(u)=\mbox{tr}_g(\tilde{\nabla} du)=\De u+A(u)(\n u,\n u)$ is the tension field and $\tilde{\n}_t=\tilde{\n}_{\p t}$, where $\tilde{\nabla}$ denotes the induced connection on the pull-back bundle $u^*TN$. In the case that $(N,J)=(\U^2, u\times)$, the above equation \eqref{eq-ASMF} is just \eqref{pra-eq0}.

Since $u$ is smooth and $\frac{\p u}{\p \nu}|_{\p\Om\times[0,T]}=0$, for any $k\in \mathbb{N}$ there holds
\[\frac{\p \p^k_t u}{\p \nu}|_{\p \Om\times [0,T]}=0,\]
and hence at $t=0$, we have
\[\frac{\p V_k}{\p \nu}|_{\p \Om}=0,\]
where we set
\[V_k(u_0)=\p^k_tu|_{t=0}.\]
In particular, $V_0=u_0$ and
\[V_1=\ep\tau(u_0)+J(u_0)\tau(u_0).\]
Moreover, one can refer to \cite{CJ,CW} for precise formula of $V_k(u_0)$ with $k>1$ in the case $(N, J)=(\U^2, u\times)$.
\medskip

On the contrary, to get very regular solution to \eqref{eq-ASMF}, we need to assume that $u_0$ satisfies the following necessary compatibility conditions on boundary.
\begin{defn}\label{def-comp}
Let $k\in \mathbb{N}$, $u_0\in H^{2k+2}(\Om,N)$. We say $u_0$ satisfies the compatibility condition of order $k$, if there holds that for any $j\in\{0,1,\dots,k\}$
\begin{equation}\label{com-cond}
\frac{\p V_j}{\p \nu}|_{\p \Om}=0.
\end{equation}
\end{defn}

Intrinsically, if we denote $$\tilde{V}_k(u_0)=\tilde{\n}^k_t u|_{t=0}\in \Ga(u^*_0(TN)),$$
then the compatibility conditions defined in \eqref{com-cond} has the below equivalent characterization.
\begin{prop}\label{intrc-cd}
Let $k\in \mathbb{N}$, $u_0\in H^{2k+2}(\Om,N)$. Then $u_0$ satisfies the compatibility condition of order $k$, if and only if there holds that for any $j\in\{0,1,\dots,k\}$,
\begin{equation}\label{com-cond1}
\tilde{\n}_\nu \tilde{V}_j|_{\p \Om}=0.
\end{equation}
\end{prop}

\begin{proof}
The necessity is proved by induction on $k$. Since $V_1=\tilde{V}_1$, if we assume $\frac{\p V_1}{\p \nu}|_{\p\Om}=0$, then we have
\[\tilde{\n}_\nu \tilde{V}_1|_{\p\Om}=\frac{\p \tilde{V}_1}{\p \nu}|_{\p\Om}+A(u_0)(\frac{\p u_0}{\p \nu}|_{\p\Om}, \tilde{V}_1)=0,\]
where $A(\cdot, \cdot)$ is the second fundamental form of $N$ in $\Real^K$. Then, we assume that the result is true for $1\leq l\leq k-1$. For the case $l=k\geq 2$, by definition of $\tilde{V}_k$, we take a simple calculation to get
\begin{align*}
\tilde{V}_k=V_k+\sum_{\si}B_{\si(k)}(u_0)(V_{a_1}, \cdots, V_{a_s})
\end{align*}
where the sum is over all indices $a_1,\cdots,a_s$ such that $1\leq a_i\leq k-1$ and $a_1+\cdots+a_s=k$,
\[(a_1, \cdots, a_s)=\si(k)\]
is a partition of $k$, and each $B_{\si(k)}$ is a multi-linear vector valued function on $\Real^K$. For more details on the above calculations we refer to the page 1451 in \cite{DW}.
	
Hence, by using the assumption of induction, we have
\begin{align*}
\tilde{\n}_\nu \tilde{V}_k|_{\p\Om}=&\frac{\p \tilde{V}_k}{\p \nu}|_{\p\Om}+A(u_0)(\frac{\p u_0}{\p \nu},\tilde{V}_k)|_{\p\Om}\\
=&\frac{\p V_k}{\p \nu}|_{\p\Om}+\sum_{\si}\n B_{\si(a)}(u_0)(\frac{\p u_0}{\p \nu}|_{\p \Om}, V_{a_1}, \cdots, V_{a_s})\\
=&0.
\end{align*}
	
On the contrary, the proof is almost the same as in the above, so we omit it.
\end{proof}

\subsection{Uniform Compatibility conditions of the initial data}\label{uni-cc}
In this part, we show that the compatibility conditions given in \eqref{com-cond}(or \eqref{com-cond1}) are actually independent of $\ep$ in the following cases (see Propositions \ref{intrc-cd1} and \ref{intrc-cd2} ).

By using the equation
\[\tilde{\n}_tu=\ep\tau(u)+J(u)\tau(u)\]
and the fact $\tilde{\n} J=0$ since $(N, J)$ is a K\"ahler manifold, first of all we get a useful equivalent characterization of the 1-order compatibility conditions in Definition \ref{intrc-cd}. It is not difficult to show that
\[\tilde{V}_1=\ep \tau(u_0)+J(u_0)\tau(u_0).\]
Thus, there holds
\[\tilde{\n}_\nu\tilde{V}_1|_{\p\Om}=\ep\tilde{\n}_\nu \tau(u_0)|_{\p\Om}+J(u_0)\tilde{\n}_\nu \tau(u_0)|_{\p\Om}\]
since $\tilde{\n} J=0$. Therefore, $\tilde{\n}_\nu\tilde{V}_1|_{\p\Om}=0$ if only if $\tilde{\n}_\nu \tau(u_0)|_{\p\Om}=0$. Namely, we have the following
\begin{prop}\label{intrc-cd1}
Let $u_0\in H^{4}(\Om,N)$. Then $u_0$ satisfies the compatibility condition of order $1$, if and only if there holds
\begin{equation}\label{com-cond2}
\tilde{\n}_\nu u_0|_{\p \Om}=0\quad\text{and}\quad\tilde{\n}_\nu \tau(u_0)|_{\p \Om}=0.
\end{equation}
\end{prop}
\begin{rem}
The compatibility condition \eqref{com-cond2} is independent of $\ep$, which implies that equation \eqref{S-eq} and its parabolic perturbed equation \eqref{pra-eq0} share the same 1-order compatibility conditions of the initial data.
\end{rem}

Secondly, when $\Om=\mathring{I}=(0,1)\subset \Real^1$ is the interval from $0$ to 1 with coordinate $\{x\}$, we can also get an equivalent characterization of the $k$-order compatibility conditions with $k\geq 1$. For any $k\in \mathbb{N}$, denoting
\[W_{k}(u_0)=\tilde{\n}^{2k}_xu_0\]
and setting $I=[0,1]$, we have the following

\begin{prop}\label{intrc-cd2}
Let $k\in \mathbb{N}$, $u_0\in H^{2k+2}(I,N)$. Then $u_0$ satisfies the compatibility condition of order $k$, if and only if there holds that for any $j\in\{0,1,\dots,k\}$,
\begin{equation}\label{com-cond3}
\tilde{\n}_x W_j|_{\p I}=\tilde{\n}^{2j+1}_xu_0|_{\p I}=0.
\end{equation}
\end{prop}

To prove this proposition, we need to show the following basic formula. For any $1\leq l\leq k+1$, since $\tilde{\n} J=0$, a simple calculation gives
\begin{equation}\label{F1}
\begin{aligned}
\tilde{\n}^{l}_t u=\ep\tilde{\n}_x\tilde{\n}_x\tilde{\n}^{l-1}_tu+J\tilde{\n}_x\tilde{\n}_x\tilde{\n}^{l-1}_tu+Q(\tilde{\n}_t u).
\end{aligned}
\end{equation}
Here $Q=0$ for $l=1$, and for $l\geq 2$, we have
\begin{align*}
Q(\tilde{\n}_t u)=&\sum_{\si}Q_{\si(l-1)}(u)(\tilde{\n}^{a_1}_t u, \cdots, \tilde{\n}^{a_{s-2}}_t u, \tilde{\n}_x\tilde{\n}^{a_{s-1}}_t u,\tilde{\n}_x\tilde{\n}^{a_{s}}_t u)\\
&+\sum_{\si}Q_{\si(l)}(u)(\tilde{\n}^{b_1}_t u, \cdots, \tilde{\n}^{b_{r}}_t u)
\end{align*}
where $1\leq a_i\leq l-1$ for $1\leq i\leq s-2$, $a_1+\cdots+a_s=l-1$,
\[(a_1, \cdots, a_s)=\si(l-1)\]
is a partition of $l-1$; $1\leq b_j\leq l-2$ for $1\leq j\leq r$ with $r\geq 3$, and $b_1+\cdots+b_r=l$,\[(b_1, \cdots, b_r)=\si(l)\]
is a partition of $l$; and $Q$ is a multi-linear functional on $u^*(TN)$. Here we have used the fact
\[\tilde{\n}_x\tilde{\n}_x u=\frac{1}{1+\ep^2}(\ep\tilde{\n}_t u-J\tilde{\n}_t u).\]

So, taking $t=0$ yields
\begin{equation}\label{Key-eq}
\begin{aligned}
\tilde{V}_l=\ep \tilde{\n}_x\tilde{\n}_x\tilde{V}_{l-1}+J\tilde{\n}_x\tilde{\n}_x\tilde{V}_{l-1}+Q_{l-1}
\end{aligned}
\end{equation}
where $Q_0=0$, and for $l\geq 2$,
\begin{align*}
Q_{l-1}=&\sum_{\si}Q_{\si(l-1)}(u_0)(\tilde{V}_{a_1}, \cdots,\tilde{V}_{a_{s-2}}, \tilde{\n}_x\tilde{V}_{a_{s-1}},\tilde{\n}_x\tilde{V}_{a_s})\\
&+\sum_{\si}Q_{\si(l)}(u_0)(\tilde{V}_{b_1}, \cdots, \tilde{V}_{b_r})
\end{align*}

Taking derivatives with respect to $x$ on both sides of the above equation \eqref{Key-eq} and assuming $\tilde{\n}_x\tilde{V}_{q}|_{\p I}=0$ for $q\leq l-1$, we get
\[\tilde{\n}_x\tilde{V}_l|_{\p I}=\ep\tilde{\n}_x\tilde{\n}_x\tilde{\n}_x\tilde{V}_{l-1}|_{\p I}+J\tilde{\n}_x\tilde{\n}_x\tilde{\n}_x\tilde{V}_{l-1}|_{\p I}+\tilde{\n}_x Q_{l-1}|_{\p I}\]
where
\begin{align*}
\tilde{\n}_x Q_{l-1}|_{\p I}=&\sum_{\si}Q_{\si(a)}(u_0)(\tilde{V}_{a_1}, \cdots,\tilde{V}_{a_{s-2}}, \tilde{\n}_x\tilde{V}_{a_{s-1}},\tilde{\n}_x\tilde{\n}_x\tilde{V}_{a_s})|_{\p I}\\
&+\sum_{\si}Q_{\si(l)}(u_0)(\tilde{V}_{b_1}, \cdots, \tilde{\n}_x\tilde{V}_{b_i}, \cdots, \tilde{V}_{b_r})|_{\p I}=0,
\end{align*}
since $0\leq a_i, b_i\leq l-1$. Thus, $\tilde{\n}_x\tilde{V}_l|_{\p  I}=0$ is equivalent to
\[\tilde{\n}^3_x\tilde{V}_{l-1}|_{\p I}=0.\]

Then by using the method of induction, we have the following result.
\begin{lem}\label{Key-estimate}
Suppose that
	\[\tilde{\n}^{2j+1}_x u_0|_{\p I}=0\]
for any $0\leq j\leq k$, and
	\[\tilde{\n}_x\tilde{V}_j|_{\p I}=0\]
for any $0\leq j\leq k+1$. Then for $1\leq s\leq l-1$ with $1\leq l\leq k+1$, there holds
\begin{itemize}
\item[$(1)$] for any $0\leq q\leq s$, $\tilde{\n}^{2q+1}_x\tilde{V}_{l-s}|_{\p I}=0$;
\item[$(2)$] for any $1\leq q\leq s$, $\tilde{\n}^{2q-1}_xQ_{l-s}|_{\p I}=0$.
\end{itemize}
\end{lem}
\begin{proof}
We show this result by inducting on $l$. In the above, we have shown this results hold true in the case that $l=2$ (such that $s=1$). Next, we assume that the results hold for any $l\leq l_0$  with  $2\leq l_0\leq k+1$, then we intend to prove the results in the case of $l=l_0+1\leq k+1$.
	
To show the results in the case that $l=l_0+1\leq k+1$, we apply again the method of induction on $s$. For $s=1$, the desired result has been established in above, then we assume that for $s \leq l-1=l_0$, the results are true. In particular, we have
\[\tilde{\n}^{2s+1}_{x}\tilde{V}_{l-s}|_{\p I}=0.\]
	
Next, it remains to show the case that $s+1\leq l-1=l_0$. Noting that we can use the assumption of induction with $l_0=l-1$, since $1\leq s\leq l_0-1$. Then for any $1\leq q\leq s$ there holds
\[\tilde{\n}^{2q+1}_x\tilde{V}_{l_0-s}|_{\p I}=0,\quad \tilde{\n}^{2q-1}_xQ_{l_0-s}|_{\p I}=0.\]
Thus, it remains to show
\[\tilde{\n}^{2(s+1)+1}_x\tilde{V}_{l-s-1}|_{\p I}=0,\quad \tilde{\n}^{2s+1}_xQ_{l-s-1}|_{\p I}=0.\]
	
To this end, a simple computation gives
\begin{align*}
\tilde{\n}^{2s+1}_x\tilde{V}_{l-s}|_{\p I}=&\ep\tilde{\n}^{2(s+1)+1}_x\tilde{V}_{l-s-1}|_{\p\Om}+J\tilde{\n}^{2(s+1)+1}_x\tilde{V}_{l-s-1}|_{\p I}\\
&+\tilde{\n}^{2s+1}_x Q_{l_0-s}|_{\p I}.
\end{align*}	
Here,
\begin{align*}
\tilde{\n}^{2s+1}_x Q_{l_0-s}|_{\p I}=&\tilde{\n}^{2s+1}_x\{\sum_{\si}Q_{\si(l_0-s)}(u_0)(\tilde{V}_{a_1}, \cdots,\tilde{V}_{a_{s'-2}}, \tilde{\n}_x\tilde{V}_{a_{s'-1}},\tilde{\n}_x\tilde{V}_{a_{s'}})\}\\
&+\tilde{\n}^{2s+1}_x\{\sum_{\si}Q_{\si(l_0-s+1)}(u_0)(\tilde{V}_{b_1}, \cdots, \tilde{V}_{b_{r'}})\}
\end{align*}
where $1\leq a_i\leq l_0-s$ for $1\leq i\leq s'-2$, and $a_1+\cdots+a_{s'}=l_0-s$; $1\leq b_i\leq l_0-s-1$ for $1\leq i\leq r'$, and $b_1+\cdots+b_{r'}=l_0-s+1$.
	
\medskip
Now, we claim that $\tilde{\n}^{2s+1}_x Q_{l_0-s}|_{\p I}=0$. A direct calculation shows
\begin{align*}
&\tilde{\n}^{2s+1}_x Q_{l_0-s}|_{\p I}\\
=&\sum_{j_0+\cdots+j_{s'}=2s+1}\sum_{\si}\tilde{\n}^{j_0}_xQ_{\si(l_0-s)}(u_0)(\tilde{\n}^{j_1}_x\tilde{V}_{a_1}, \cdots,\tilde{\n}^{j_{s'-2}}_x\tilde{V}_{a_{s'-2}}, \tilde{\n}^{j'_{s'-1}}_x\tilde{V}_{a_{s'-1}},\tilde{\n}^{j'_{s'}}_x\tilde{V}_{a_s'})|_{\p I}\\
&+\sum_{i_0+\cdots+i_{r'}=2s+1}\sum_{\si}\tilde{\n}^{i_0}_xQ_{\si(l_0-s+1)}(u_0)(\tilde{\n}^{i_1}_x\tilde{V}_{b_1}, \cdots\tilde{\n}^{i_{r'}}_x\tilde{V}_{b_{r'}})|_{\p I}.
\end{align*}
	
For simplicity, we denote $j'_{s'-1}=j_{s'-1}+1$ and $j'_{s'}=j_{s'}+1$. Since \[j_0+\cdots+j_{s'-2}+j_{s'-1}+j_{s'}=2s+1\]
is odd, then there exits at least one odd $j_q$ with $q\leq 2s+1$ in $ \{j_0,\cdots, j_{s'-2},j'_{s'-1},j'_{s'}\}$. By the assumption of induction on $l$ with $l\leq l_0$, we have
\[\tilde{\n}^{j_q}\tilde{V}_{j_q}|_{\p I}=0.\]
This implies
\[\sum_{j_0+\cdots+j_{s'}=2s+1}\sum_{\si}\tilde{\n}^{j_0}_xQ_{\si(l_0-s)}(u_0)(\tilde{\n}^{j_1}_x\tilde{V}_{a_1}, \cdots,\tilde{\n}^{j_{s'-2}}_x\tilde{V}_{s'-2}, \tilde{\n}^{j'_{s'-1}}_x\tilde{V}_{a_{s'-1}},\tilde{\n}^{j'_{s'}}_x\tilde{V}_{a_s'})|_{\p I}=0,\]
since it is not difficult to show $\tilde{\n}^{j_0}_xQ_{\si(l_0-s)}(u_0)|_{\p I}=0$ if $j_0$ is odd.
	
By similar arguments with that in the above, we can also show
\[\sum_{i_0+\cdots+i_{r'}=2s+1}\sum_{\si}\tilde{\n}^{i_0}_xQ_{\si(l_0-s+1)}(u_0)(\tilde{\n}^{i_1}_x\tilde{V}_{b_1}, \cdots,\tilde{\n}^{i_{r'}}_x\tilde{V}_{b_{r'}})|_{\p I}=0.\]
	
So, there holds
\[\tilde{\n}_x^{2s+1}\tilde{V}_{l-s}|_{\p I}=\ep\tilde{\n}_x^{2(s+1)+1}\tilde{V}_{l-s-1}|_{\p I}+J\tilde{\n}_x^{2(s+1)+1}\tilde{V}_{l-s-1}|_{\p I}=0\]
and
\[\tilde{\n}_x^{2s+1}Q_{l-s-1}|_{\p I}=0.\]
Immediately it follows that
\[\tilde{\n}_x^{2(s+1)+1}\tilde{V}_{l-s-1}|_{\p  I}=0.\]
Therefore, the proof is completed.
\end{proof}

\begin{rem}\label{Key-estimate1}
In fact, we only need to assume that
\[\tilde{\n}_x\tilde{V}_j|_{\p I}=0,\quad \tilde{\n}^{2j+1}_xu_0|_{\p I}=0\]
with $0\leq j\leq k$,  then, by taking the same argument as in the proof of Lemma \ref{Key-estimate}, we can show that there holds true for any $0\leq q\leq s\leq k-1$
	\[\tilde{\n}^{2q+1}_xQ_{k-s}|_{\p I}=0.\]
\end{rem}

\medskip
Now we are in the position to show Proposition \ref{intrc-cd1}.

\begin{proof}
	
The proof of this proposition is divided into three steps.
	
\medskip
\noindent\emph{\textbf{Step 1: In the case that $k=0,1$.}}\
	
When $k=0$, we have $\tilde{V}_0=W_0=u_0$ and obviously the result is true since $\frac{\p u_0}{\p x}|_{\p\Om}=0$. In the case that $k=1$, we have
\[\tilde{V}_1=\ep\tilde{\n}_x\tilde{\n}_x u_0+J\tilde{\n}_x\tilde{\n}_x u_0=\ep W_1+JW_1,\]
it follows that
\[\tilde{\n}_x\tilde{V}_1|_{\p I}=0\]
if and only if
\[\tilde{\n}_x W_1|_{\p I}=0.\]
	
Next we show the general case by induction on $k$. Assume that this proposition has been established for the case of order less and equal than $k$. Now we need to show the result also holds in the case of $k+1$.
	
\medskip
\noindent\emph{\textbf{Step 2: The general case (i.e. $k\geq 1$)(From \eqref{com-cond1} to \eqref{com-cond2}).}}\
We assume that $\tilde{\n}_x \tilde{V}_{l}|_{\p I}=0$ for any $l\leq k+1$. Then, by assumption of induction we have
\[\tilde{\n}_x W_{l}|_{\p I}=\tilde{\n}^{2l+1}_xu_0|_{\p I}=0\]
for any $l\leq k$. Thus, it remains to show $\tilde{\n}_x W_{k+1}|_{\p I}=0$.
	
According to Lemma \ref{Key-estimate}, by taking $l=k+1$ and $s=q=k$, we get
\[0=\tilde{\n}_x^{2k+1}\tilde{V}_1|_{\p I}=\ep \tilde{\n}_x W^{k+1}|_{\p I}+J\tilde{\n}_x W^{k+1}|_{\p I},\]
which gives
\[\tilde{\n}_x W^{k+1}|_{\p I}=0.\]

\medskip
\noindent\emph{\textbf{Step 3: The general case (i.e. $k\geq 1$)(From \eqref{com-cond2} to \eqref{com-cond1}).}}\
We assume that $\tilde{\n}_xW_{l}|_{\p I}=0$ for any $l\leq k+1$. Then, by the assumption of induction we have
	\[\tilde{\n}_x\tilde{V}_l|_{\p I}=0\]
for any $l\leq k$. Next, we prove $\tilde{\n}_x\tilde{V}_{k+1}|_{\p I}=0$.
	
A simple calculation gives
\[\tilde{V}_{k+1}=\ep\tilde{\n}^2_x\tilde{V}_k+J\tilde{\n}^2_x\tilde{V}_k+Q_{k}.\]
Since $\tilde{\n}_xQ_k|_{\p I}=0$, this implies
\[\tilde{\n}_x\tilde{V}_{k+1}|_{\p I}=\ep\tilde{\n}^3_x\tilde{V}_k|_{\p I}+J\tilde{\n}^3_x\tilde{V}_k|_{\p I}.\]
So, to show $\tilde{\n}_x\tilde{V}_{k+1}|_{\p I}=0$, we only need to show $\tilde{\n}^3_x\tilde{V}_k|_{\p I}=0$.
	
\medskip
On the other hand, by the estimates in Remark \ref{Key-estimate1} we have
\[\tilde{\n}^{2(k-s)+1}_xQ_{s}|_{\p I}=0\]
for $1\leq s\leq k$.
Therefore, a direct calculation shows
\begin{align*}
\n^3_x\tilde{V}_k|_{\p I}=&\ep\tilde{\n}^5_x\tilde{V}_{k-1}|_{\p I}+J\tilde{\n}^5_x\tilde{V}_{k-1}|_{\p I}+\tilde{\n}^{3}_xQ_{k-1}|_{\p I}\\
=&\ep\tilde{\n}^5_x\tilde{V}_{k-1}|_{\p I}+J\tilde{\n}^5_x\tilde{V}_{k-1}|_{\p I}.
\end{align*}
Thus, we only need to show $\tilde{\n}^5_x\tilde{V}_{k-1}|_{\p I}=0$. By repeating the above process with $k$ steps, we can see that in order to show $\tilde{\n}_x\tilde{V}_{k+1}|_{\p I}=0$ one only need to derive $\tilde{\n}^{2k+1}_x\tilde{V}_{1}|_{\p I}=0$.
	
Since
\[\tilde{\n}^{2k+1}_x\tilde{V}_{1}|_{\p I}=\ep\tilde{\n}_x W^{k+1}|_{\p I}+J\tilde{\n}_x W^{k+1}|_{\p I},\]
the fact $\tilde{\n}_x W^{k+1}|_{\p I}=0$ implies
\[\tilde{\n}_x\tilde{V}_{k+1}|_{\p I}=0.\]
Therefore, we finish the proof.
\end{proof}

\begin{rem}
The compatibility condition \eqref{com-cond3} is independent of $\ep$.
\end{rem}

However, when the dimension of $\Om$ is larger than 1, the $k$-order (with $k>1$) compatibility conditions defined in \eqref{com-cond}(or \eqref{com-cond1}) seem to be dependent of $\ep$. To proceed, we need to add some stronger conditions on $u_0$ to guarantee that equation \eqref{S-eq} and its parabolic perturbed equation \eqref{pra-eq0} share the same $k$-order compatibility conditions (a similar compatibility condition for parabolic perturbed equation appears in \cite{FT}).
\begin{defn}\label{stronger-com-cond}
Let $1<k\in \mathbb{N}$, $u_0\in H^{2k+2}(\Om,N)$. We say $u_0$ satisfies a stronger compatibility condition of order $k$ (denoted by $CC(k)$), if we have that for any $1\leq j\leq 2k$ there hold true
\begin{equation}\label{s-com-cond}
\frac{\p}{\p \nu}\p^ju_0|_{\p\Om}=0,
\end{equation}
where $\p^ju_0=\(\frac{\p^ju_0}{\p x^{i_1}\cdots\p x^{i_j}}\)$ are all the $j$-th partial derivatives of $u_0$.
\end{defn}

In fact, there always exists an initial data $u_0$ satisfying the compatibility condition defined in \eqref{s-com-cond}. For instance, we can choose a smooth map $u_0:\Omega\to N$ such that $u_0$ is constant in a neighborhood of $\p\Om$.

Meanwhile, it should be pointed out that these stronger compatibility conditions defined in \eqref{s-com-cond} imply \eqref{com-cond} and \eqref{com-cond1}. Concretely speaking, for any $u\in N$, let $P(u): \Real^K\to T_uN$ be the standard projection operator. Then $\tau(u)=P(u)\De u$, and hence Equation \eqref{eq-ASMF} has the following extrinsic form
\[\p_tu=\ep(\De u+A(u)(\n u,\n u))+\tilde{J}(u)\De u,\]
where for simplicity we denote $J(u)P(u)$ by $\tilde{J}(u)$. Thus, for any $k\in \mathbb{N}$, by applying this extrinsic equation we take a simple calculation to see
\begin{equation}\label{eq-V_k}
\begin{aligned}
V_{k+1}=&\ep\De V_k + \tilde{J}(u_0)\De V_k+2\ep A(u_0)(\n V_k, \n u_0)\\
	&+\ep \n A(u_0)(V_k, \n u_0, \n u_0)+\n \tilde{J}(u_0)(V_k, \De u_0)\\
	&+\ep\sum_{i_1+\cdots+i_s+m+l=k,\,\,1\leq i_j<k}\n^sA(u_0)(V_{i_1},\cdots, V_{i_s}, \n V_m, \n V_l)\\
	&+\sum_{i_1+\cdots+i_s+m=k,\,\,1\leq i_j<k}\n^s \tilde{J}(u_0)(V_{i_1},\cdots, V_{i_s}, \De V_m).	
\end{aligned}
\end{equation}

In particular, we have
\[V_1=\ep(\De u_0+A(u_0)(\n u_0, \n u_0))+\tilde{J}(u_0)\De u_0.\]
So, obviously the $1$-order compatibility condition of $u_0$ defined in \eqref{s-com-cond} (namely
$\frac{\p u_0}{\p\nu}|_{\p\Om}=0$, $\frac{\p}{\p\nu}\p u_0|_{\p\Om}=0$ and $\frac{\p }{\p\nu}\p^2 u_0|_{\p\Om}=0$) implies
\[\frac{\p}{\p \nu}V_1|_{\p\Om}=0,\]
that is the 1-order compatibility condition of $u_0$ defined \eqref{com-cond}.

In the case of $k\geq 2$, by applying Formula \eqref{eq-V_k}, we have
\[V_k=\sum_{\si}\tilde{B}_{\si(2k)}(u_0)(\p^{j_1} u_0, \cdots, \p^{j_s}u_0)\]
where the sum is over all indices $j_1,\cdots,j_s$ such that $1\leq j_i\leq 2k$ and $j_1+\cdots+j_s=2k$,
\[(j_1, \cdots, j_s)=\si(2k)\]
is a partition of $2k$, and each $\tilde{B}_{\si(2k)}$ is a multi-linear vector valued function on $\Real^K$. Therefore, it is not difficult to show that the $k$-order compatibility condition defined in \eqref{s-com-cond} implies \eqref{com-cond} and \eqref{com-cond1}.

\subsection{Another compatibility conditions}\label{another-cc}
We also need to make use of the following conclusions on compatibility conditions to vanish the boundary term in the process of energy estimates in the coming sections.
\begin{prop}\label{comp-cond2}
Let $\Om$ be a smooth bounded domain in $\Real^m$ with $m\geq 1$, $u: \Om\times [0,T]\to \Real$ be a map satisfying
\[\p_t^{i}u\in L^2([0,T], H^{2k-2i}(\Om))\]
for any $0\leq i< k$, where $k\geq 1$. If $$\frac{\p u}{\p \nu}|_{\p\Om\times[0,T]}=0$$ in the sense of trace,
then, for $0\leq j<k$ there hold true
\[\frac{\p}{\p \nu}\p^j_tu|_{\p\Om\times[0,T]}=0.\]
\end{prop}

\begin{proof}
Without loss of generality, we assume $k\geq 2$. We only need to show
\[\frac{\p}{\p \nu}\p_tu|_{\p\Om\times[0,T]}=0\]
in the sense of trace, the remaining cases can be dealt with by some almost the same arguments.

Let $\phi\in C^\infty(\bar{\Om}\times [0,T])$. It is easy to see that there holds true
\begin{equation}\label{comp}
\int_{0}^{T}\int_{\Om}\<\De u, \p_t \phi\>dxdt=-\int_{0}^{T}\int_{\Om}\<\n u, \p_t\n \phi\>dxdt,
\end{equation}
since
$$\frac{\p u}{\p \nu}|_{\p\Om\times[0,T]}=0.$$
Then, a simple calculation shows
\begin{align*}
	\mbox{LHS of \eqref{comp}} =&-\int_{0}^{T}\int_{\Om}\<\p_t\De u, \phi\>dxdt+\int_{\Om}\<\De u, \phi\>dx(T)\\
	&-\int_{\Om}\<\De u, \phi\>dx(0)\\
	=&-\int_{0}^{T}\int_{\Om}\<\p_t\De u, \phi\>dxdt-\int_{\Om}\<\n u, \n\phi\>dx(T)\\
	&+\int_{\Om}\<\n u, \n\phi\>dx(0)
\end{align*}
and
\begin{align*}
	\mbox{RHS of \eqref{comp}} =&-\int_{0}^{T}\int_{\Om}\<\n u, \p_t\n \phi\>dxdt\\
	=&\int_{0}^{T}\int_{\Om}\<\n \p_tu, \n \phi\>dxdt-\int_{\Om}\<\n u, \n\phi\>dx(T)\\
	&+\int_{\Om}\<\n u, \n\phi\>dx(0).
\end{align*}
It follows
\[\int_{0}^{T}\int_{\Om}\<\De \p_tu, \phi\>dxdt=-\int_{0}^{T}\int_{\Om}\<\n \p_tu, \n \phi\>dxdt.\]
This is just what we want to prove. Here we have used Lemma \ref{C^0-em} which tells us that
\[u\in C^0([0,T],H^3(\Om)).\]
Hence, if we take $\phi(x,t)=\eta(t)f(x)$, then
\begin{align*}
\int_{0}^{T}\(\int_{\Om}\<\De u, f\>dx+\int_{\Om}\<\n u, \n f\>dx\)\eta(t)dt=0.
\end{align*}
This implies
\[\int_{\Om}\<\De u, f\>dx=-\int_{\Om}\<\n u, \n f\>dx\]
for any $t\in [0,T]$.
\end{proof}

\section{$H^5$-regular local solution}\label{s: H^5}\

Let $u_0\in H^5(\Om, \U^2)$, satisfying the $1$-order compatibility condition (see \eqref{com-cond2}). We consider the parabolic perturbed equation (i.e. Landau-Lifshitz-Gilbert equation)
\begin{equation}\label{p-sch}
	\begin{cases}
		\p_tu =\ep \tau(u)+u\times \De u\quad\quad&\text{(x,t)}\in\Om\times \Real^+,\\[1ex]
		\frac{\p u}{\p \nu}=0&\text{(x,t)}\in\p\Om\times \Real^+,\\[1ex]
		u(x,0)=u_0: \Om\to \U^2,
	\end{cases}
\end{equation}
with $\ep\in (0,1)$.

Recall that we have established the following theorem in \cite{CW,CW1}(also see \cite{CJ}).
\begin{thm}\label{para-H^5}
Suppose that $u_0\in H^5(\Om, \U^2)$, and satisfies the $1$-order compatibility condition
\[\tilde{\n}_\nu u_0|_{\p \Om}=0\quad\text{and}\quad\tilde{\n}_\nu \tau(u_0)|_{\p \Om}=0,\]
i.e. \eqref{com-cond2} in Proposition \ref{intrc-cd1}. Then there exists a positive time $T_\ep$ depending only on $\ep$ and $\norm{u_0}_{H^2(\Om)}$ such that  equation \eqref{p-sch} admits a unique regular solution $u_\ep$, which satisfies for any $T<T_\ep$ that
\[\p^i_tu_\ep\in L^\infty([0,T], H^{5-2i}(\Om))\cap L^2([0,T], H^{6-2i}(\Om)).\]
for $0 \leq i\leq 2$.

Moreover, there exists a uniform positive number $T_0<T_\ep$ depending only on $\norm{u_0}_{H^3}$, such that $u_\ep$ satisfies
\begin{equation}\label{unif-es}
\sup_{0\leq T<T_0}(\norm{u_\ep}^2_{H^3(\Om)}+\norm{\frac{\p u_\ep}{\p t}}_{H^1(\Om)})\leq C(\norm{u_0}_{H^3}).
\end{equation}
\end{thm}
\begin{proof}
We can apply Theorem 3.1 in \cite{CW1} to conclude that there exists a maximal existence time $T_\ep$ (depending only on $\ep$ and $\norm{u_0}_{H^2(\Om)}$) such that equation \eqref{p-sch} admits a unique regular solution $u_\ep$, which satisfies that for any $T<T_\ep$
\begin{itemize}
\item[$(1)$] $u_\ep\in L^\infty([0,T],H^3(\Om, \U^2))\cap L^{2}([0,T],H^4(\Om, \U^2))$;
\item[$(2)$] $\frac{\p u_\ep}{\p t}\in L^\infty([0,T],H^1(\Om))\cap L^{2}([0,T],H^2(\Om))$ and $\frac{\p^2 u_\ep}{\p t^2}\in L^{2}([0,T],L^2(\Om))$.
\end{itemize}
Moreover, by Theorem 1.1 in \cite{CW1}, there exists a positive number $T_0<T_\ep$ depending only on $\norm{u_0}_{H^3}$, such that $u_\ep$ satisfies
\[\sup_{0\leq T<T_0}(\norm{u_\ep}^2_{H^3(\Om)}+\norm{\frac{\p u_\ep}{\p t}}_{H^1(\Om)})\leq C(\norm{u_0}_{H^3}).\]

On the other hand, since $u_0\in H^5(\Om, \U^2)$ satisfying the $1$-order compatibility condition, by applying Theorem 1.3 in \cite{CW} to improve the regularity of $u_\ep$, we get the desired estimates of $u_\ep$ in this theorem.
\end{proof}

\medskip
Next, we follow a similar argument with that in \cite{CW1} to show the uniform $H^3$-estimates of
$$v=\p_tu_\ep.$$
Then, by using equation \eqref{p-sch} again, we can improve the uniform estimates of $u_\ep$ such that the sequence $\{u_\ep\}$ is uniform bounded $L^\infty([0,T_1], H^5(\Om))$ for some uniform positive number $T_1\leq T_0$. Thus, we obtain the desired $H^5$-regular solution to \eqref{S-eq} by letting $\ep\to 0$.
\subsection{Uniform $H^2$-estimates}\
First of all, we show a uniform $H^2$-estimate of $v$ by directed energy estimates.
Theorem \ref{para-H^5} implies that there holds true
\[v\in L^\infty([0,T], H^{3}(\Om))\cap L^2([0,T], H^{4}(\Om))\]
for any $0<T<T_\ep$, and
\[\sup_{0\leq T<T_0}\norm{v}^2_{H^1(\Om)}\leq C(\norm{u_0}_{H^3})\]
for $0<T_0<T_\ep$.

Let $$w=\p^2_t u_\ep.$$
It belongs to the space $L^\infty([0,T], H^{1}(\Om))\cap L^2([0,T], H^{2}(\Om))$, and satisfies the following equation
\begin{equation}\label{eq-w}
\begin{cases}
\p_t w =\ep \De w+u_\ep\times \De w+2\ep \n w\cdot \n u_\ep u_\ep+\ep |\n u_\ep|^2w+w\times \De u_\ep+f(u_\ep, v),\\[1ex]
\frac{\p w}{\p \nu}|_{\p\Om}=0, \\[1ex]
w(x,0)=V_2(u_0),
\end{cases}
\end{equation}
where
\[f(u_\ep,v)=4\ep\n v\cdot \n u_\ep v+2\ep|\n v|^2u_\ep+2v\times \De v.\]

By taking $w$ as a test function to equation \eqref{eq-w}, we can derive the following
\begin{equation}\label{L^2-w}
\begin{aligned}
\frac{1}{2}\frac{\p }{\p t}\int_{\Om} |w|^2dx+\ep\int_{\Om}|\n w|^2dx=&\int_{\Om}\<u_\ep\times \De w,w\>dx+\ep\int_{\Om}\<|\n u_\ep|^2w,w\>dx\\
&+2\ep\int_{\Om}\<\n w\cdot \n u_\ep u_\ep,w\>dx+\int_{\Om}\<f(u_\ep,v),w\>dx\\
=&I+II+III+IV.
\end{aligned}
\end{equation}
Here we have used Proposition \ref{comp-cond2} to deduce
$$\frac{\p \om}{\p \nu}|_{\p \Om\times[0,T]}=0.$$

Next, we estimate the above I, II, III and IV term by term.
\begin{align*}
|I|=&\left|\int_{\Om}\<\mbox{div}(u_\ep\times \n w),w\>-\<\n u_\ep\times \n w, w\>dx\right|\\
\leq&\int_{\Om}|\n u_\ep||\n w||w|dx\\
\leq &C\norm{u_\ep}_{H^3}\int_{\Om}|w|^2+|\n w|^2dx,\\
|II|=&\ep\left|\int_{\Om}\<|\n u_\ep|^2w,w\>dx\right|\leq C\ep\norm{u_\ep}^2_{H^3}\int_{\Om}|w|^2dx,\\
|III|=&2\ep\left|\int_{\Om}\<\n w\cdot \n u_\ep u_\ep,w\>dx\right|\leq C\ep \norm{u_\ep}^2_{H^3}\int_{\Om}|w|^2+\frac{\ep}{4}\int_{\Om}|\n w|^2dx.
\end{align*}
Here we have used the fact
\[\int_{\Om}\<\mbox{div}(u_\ep\times \n w),w\>dx=-\int_{\Om}\<u_\ep\times \n w,\n w\>dx=0\]
and the Sobolev embedding inequality
\[|\n u_\ep|_{L^\infty}\leq C\norm{u_\ep}_{H^3(\Om)}.\]

For the last term, we have
\begin{align*}
|IV|\leq &4\ep\left|\int_{\Om}\<\n v\cdot \n u_\ep v,w\>dx\right|+2\ep\left|\int_{\Om}\<|\n v|^2u_\ep,w\>dx\right| + 2\left|\int_{\Om}\<v\times \De v,w\>dx\right|\\
=&a+b+c\\
\leq & C\norm{v}^2_{H^1}(\ep \norm{u_\ep}^2_{H^3}\norm{v}^2_{H^1}+\norm{v}^2_{H^2})+C\norm{w}^2_{H^1},
\end{align*}
where
\begin{align*}
|a|\leq &4\ep \int_{\Om}|\n v||\n u_\ep|| v||w|dx\\
\leq &C\ep\norm{u_\ep}_{H^3}\norm{\n v}_{L^2}\norm{v}_{L^3}\norm{w}_{L^6}\\
\leq &C\ep \norm{u_\ep}^2_{H^3}\norm{v}^4_{H^1}+\frac{\ep}{4}\norm{w}^2_{H^1},\\
|b|\leq &\ep \norm{\n v}_{L^2}\norm{\n v}_{L^3}\norm{w}_{L^6}\\
\leq &C\ep\norm{v}^2_{H^1}\norm{v}^2_{H^2}+\frac{\ep}{4}\norm{w}^2_{H^1},\\
|c|\leq &\int_{\Om}|\n v||v||\n w|dx\\
\leq &C \norm{v}^2_{H^1}\norm{v}^2_{H^2}+C\norm{\n w}^2_{L^2}.
\end{align*}

In order to get the desired energy bounds, we need the following estimates on equivalent norms of $\norm{v}_{H^2}$ and $\norm{v}_{H^3}$.
\begin{lem}\label{v-w}
Assume that $u_\ep$ is the solution of \eqref{pra-eq0} obtained in Theorem \ref{para-H^5}. Then there exists a constant $C$ independent of $\ep$ such that for a.e. $t\in[0,T_0]$, the following estimates hold.
\begin{align}
\norm{v}^2_{H^2(\Om)}\leq &C(\norm{u_\ep}^4_{H^3}+1)\norm{v}^2_{H^1}+C\int_{\Om}|w|^2dx,\\
\norm{v}^2_{H^3(\Om)}\leq &C(\norm{u_\ep}^2_{H^3},\norm{v}^2_{H^1})(\norm{w}^2_{H^1}+1).
\end{align}
\end{lem}
\begin{proof}
By using Equation \eqref{pra-eq0}, we can see easily that
\[\De u_\ep=-|\n u_\ep|^2u_\ep+\frac{1}{1+\ep^2}(\ep \p_tu_\ep-u_\ep\times \p_t u_\ep).\]
This leads to
\[\De v=-|\n u_\ep|^2v-2\n v\cdot\n u_\ep u_\ep+\frac{1}{1+\ep^2}(\ep w-u_\ep\times w).\]
Here, $v=\p_t u_\ep$ and $w=\p_t v$.

A direct computation shows
\begin{align*}
\int_{\Om}|\De v|^2dx\leq &C\int_{\Om}|\n v|^2|\n u_\ep|^2dx+C\int_{\Om}|\n u_\ep|^4|v|^2dx+\frac{C}{1+\ep^2}\int_{\Om}|w|^2dx\\
\leq & C\norm{u_\ep}^2_{H^3}(\norm{u_\ep}^2_{H^3}+1)\norm{v}^2_{H^1}+C\int_{\Om}|w|^2dx.
\end{align*}

On the other hand, we have
\begin{align*}
\n \De v=&\n^2 v\#\n u_\ep\# u_\ep+\n v\# \n^2 u_\ep\# u_\ep+\n v\#\n u_\ep\# \n u_\ep+2\n^2 u_\ep\#\n u_\ep\#v\\
&+|\n u_\ep|^2 \n v+\frac{1}{1+\ep^2}(\ep \n w-\n u_\ep \times w-u_\ep\times \n w).
\end{align*}
Then, we have
\begin{align*}
\int_{\Om}|\n \De v|^2dx\leq &C\int_{\Om}|\n^2 v|^2|\n u_\ep|^2dx+C\int_{\Om}|\n v|^2|\n^2 u_\ep|^2dx+C\int_{\Om}|\n^2 u_\ep|^2|\n u_\ep|^2|v|^2dx\\
&+C\int_{\Om}|\n v|^2|\n u_\ep|^4dx+\frac{C}{1+\ep^2}\int_{\Om}|\n w|^2dx+\frac{C}{(1+\ep^2)^2}\int_{\Om}|w|^2|\n u_\ep|^2dx\\
\leq& C\norm{u_\ep}^2_{H^3}\int_{\Om}|\n^2v|^2dx+C\norm{u_\ep}^2_{H^3}\norm{v}^2_{H^2}+C\norm{u_\ep}^4_{H^3}\norm{v}^2_{H^1}\\
&+C\norm{u_\ep}^4_{H^3}\int_{\Om}|\n v|^2dx+C(1+\norm{u_\ep}^2_{H^3})\int_{\Om}|\n w|^2+|w|^2dx\\
\leq&C (\norm{u_\ep}^2_{H^3}+\norm{u_\ep}^4_{H^3})\norm{v}^2_{H^2}+C(1+\norm{u_\ep}^2_{H^3})\norm{w}^2_{H^1}.
\end{align*}
This is the $L^2$-estimate of $\n \De v$.

So, by taking consideration of the fact $\frac{\p v}{\p \nu}|_{\p\Om\times[0,T]}=0$, Lemma \ref{eq-n} implies the desired results in this lemma.
\end{proof}

Therefore, by combining the above estimates of $I-IV$ with the above formula \eqref{L^2-w} and applying Lemma \ref{v-w} we have that for any $0<t\leq T_0$
\begin{equation}\label{L^2-w1}
\frac{1}{2}\frac{\p }{\p t}\int_{\Om} |w|^2dx+\ep\int_{\Om}|\n w|^2dx\leq C(\sup_{0<t\leq T_0}\norm{u_\ep}_{H^3})\(\int_{\Om}(|w|^2+|\n w|^2)dx+1\).
\end{equation}

\medskip
\subsection{Uniform $H^3$-estimates}\

In this subsection, we show a uniform $H^3$-estimate of $v=\p_t u_\ep$. By a similar argument with that in the above subsection, we choose $-\De w$ as a test function to \eqref{eq-w}. However, it seems that we cannot get the desired energy estimates directly, since the lower regularity of $w$, and hence integration by parts do not make sense.

To proceed, we need to improve the regularity of $w$ by applying the $L^2$-estimates of parabolic equation as follows. We know that $w\in L^\infty([0,T], H^{1}(\Om))\cap L^2([0,T], H^{2}(\Om))$ with $0<T<T_\ep$ and satisfies the following equation
\begin{equation}\label{eq-w1}
\begin{cases}
\p_t w =\ep \De w+u_\ep\times \De w+\tilde{f},\\[1ex]
\frac{\p w}{\p \nu}|_{\p\Om\times[0,T_\ep)}=0, \\[1ex]
w(x,0)=V_2(u_0),
\end{cases}
\end{equation}
where
\[\tilde{f}=2\ep \n w\cdot \n u_\ep u_\ep+\ep |\n u_\ep|^2w+w\times \De u_\ep+f(u_\ep, v).\]
It is not difficult to show
\[\tilde{f}\in L^2([0,T], H^{1}(\Om)).\]
Hence, the classical $L^2$-estimates of parabolic equation (also see Theorem A.1 in \cite{CW1}) tells us that
\[w\in L^2_{loc}((0,T], H^{3}(\Om))\]
and
\[\frac{\p w}{\p t}\in L^2_{loc}((0,T], H^{1}(\Om)),\]
which guarantee the integration by parts in the following process of energy estimates make sense.

By taking $\De w$ as a test function of \eqref{eq-w}, we have
\begin{equation}\label{H^1-w}
\begin{aligned}
&\frac{1}{2}\frac{\p}{\p t}\int_{\Om}|\n w|^2+\ep\int_{\Om}|\De w|^2dx\\
=&-\int_{\Om}\<w\times \De u_\ep, \De w\>dx-2\int_{\Om}\<v\times \De v, \De w\>dx\\
&-\ep\int_{\Om}\<|\n u_\ep|^2w, \De w\>dx-2\ep\int_{\Om}\<\n w, \n u_\ep\>\<u_\ep, \De w\>dx\\
&-4\ep\int_{\Om}\<\n v, \n u_\ep\>\<v, \De w\>dx-2\ep\int_{\Om}\<|\n v|^2u_\ep, \De w\>dx\\
=& I^*+ II^* +III^*+IV^*+V^*+VI^*.
\end{aligned}
\end{equation}

Then, we estimate the above six terms in \eqref{H^1-w} step by steps  as follows.
\begin{align*}
|I^*|=&\left|\int_{\Om}\<w\times \n \De u_\ep, \n w\>dx\right|\\
\leq &\norm{\n w}_{L^2}\norm{w}_{L^6}\norm{\n \De u_\ep}_{L^3}\\
\leq &C\norm{w}^2_{H^1}\norm{\n \De u_\ep}_{L^3}\\
\leq &C(\norm{u_\ep}_{H^3})(1+\norm{v}_{H^2})\norm{w}^2_{H^1}\\
\leq &C\norm{w}^2_{H^1}(1+\norm{w}_{L^2}).
\end{align*}
Here, we have used the following formula
\[\De u_\ep=\frac{1}{1+\ep^2}(\ep v-u_\ep\times v)-|\n u_\ep|^2u_\ep\]
to show
\[\norm{\n \De u_\ep}_{L^3}\leq C(\norm{u_\ep}_{H^3})(1+\norm{v}_{W^{1,3}}).\]

\begin{align*}
|II^*|=&\left|\int_{\Om}\<\n v\times \De v, \n w\>dx+\int_{\Om}\<v\times \n \De v, \n w\>dx\right|\\
\leq &\int_{\Om}|\n w|^2dx +\int_{\Om}|\n v|^2|\De v|^2dx+\int_{\Om}|v|^2|\n \De v|^2dx\\
\leq &\int_{\Om}|\n w|^2dx +C\norm{v}^2_{H^2}\norm{v}^2_{H^3}\\
\leq & C(1+\norm{w}^2_{H^1})(1+\norm{w}^2_{L^2}),\\
|III^*|=&\ep\left|\int_{\Om}\<|\n u_\ep|^2w, \De w\>dx\right|\leq C\ep\norm{u_\ep}^4_{H^3}\int_{\Om}|w|^2dx+\frac{\ep}{8}\int_{\Om}|\De w|^2dx,\\
|IV^*|=&2\ep\left|\int_{\Om}\<\n w, \n u_\ep\>\<u_\ep, \De w\>dx\right|\leq C\ep\norm{u_\ep}^2_{H^3}\int_{\Om}|\n w|^2dx+\frac{\ep}{8}\int_{\Om}|\De w|^2dx,\\
|V^*|=&4\ep\left|\int_{\Om}\<\n v, \n u_\ep\>\<v, \De w\>dx\right|\leq C\ep\norm{v}^2_{H^1}(\int_{\Om}|w|^2dx+1)+\frac{\ep}{8}\int_{\Om}|\De w|^2dx,\\
|VI^*|=&2\ep\left|\int_{\Om}\<|\n v|^2u_\ep, \De w\>dx\right|\leq C\ep\norm{v}^2_{H^1}(1+\norm{w}^2_{H^1})+\frac{\ep}{8}\int_{\Om}|\De w|^2dx.
\end{align*}
Hence, for any $0<t\leq T_0$ we have
\begin{equation}\label{H^1-w1}
\frac{1}{2}\frac{\p}{\p t}\int_{\Om}|\n w|^2+\frac{\ep}{2}\int_{\Om}|\De w|^2dx\leq C(\sup_{0<t\leq T_0}\norm{u_\ep}_{H^3})(1+\norm{w}^2_{H^1})(1+\norm{w}^2_{L^2}).
\end{equation}

By combining inequalities \eqref{L^2-w1} with \eqref{H^1-w1}, the classical comparison theorem of ODE (i.e. Corollary \ref{ode}) implies the following $H^1$-estimates of $w$, and hence we can get the uniform $H^5$-estimates of $u_\ep$ by applying equation \eqref{pra-eq0} again.

\begin{prop}\label{H^5-es}
There exists a constant $C$ and $T_1$ depending only on $\norm{u_0}_{H^5}$ such that the solution $u_\ep$  to \eqref{pra-eq0} obtained in Theorem \ref{para-H^5} satisfies the following uniform bounds
\[\sup_{0<t\leq T_1}\norm{\p^i_t u_\ep}^2_{H^{5-2i}(\Om)}\leq C \]
for $i=0, 1, 2$.
\end{prop}
\begin{proof}
Let $y(t)=\norm{w}^2_{H^1}$. Since $u_\ep$ is a solution to \eqref{pra-eq0} in Theorem \ref{para-H^5}, we have
\[w\in L^\infty([0,T], H^{1}(\Om))\cap L^2([0,T], H^{2}(\Om))\]
and
\[\frac{\p w}{\p t}\in L^2([0,T], L^2(\Om))\]
by using equation \eqref{eq-w}. Hence, Lemma \ref{C^0-em} implies
\[w\in C^0([0,T], H^1(\Om))\]
for any $0<T\leq T_0$. It follows that $y(t)$ is a continuous function on $[0,T_0]$.

On the other hand, the inequalities \eqref{L^2-w1} and \eqref{H^1-w1} tell us that $y$ satisfies the following differential inequality
\begin{equation*}
\begin{cases}
y'(t)\leq C(1+y)^2,\\[1ex]
y(0)=\norm{w}^2_{H^1}|_{t=0}=\norm{V_2}^2_{H^1}.
\end{cases}
\end{equation*}

Let $T^*>0$ be the maximal existence time of solution to the below ODE
\begin{equation*}
\begin{cases}
z'(t)=C(1+z)^2,\\[1ex]
z(0)=\norm{V_2}^2_{H^1},
\end{cases}
\end{equation*}
which only depends on $\norm{V_2}^2_{H^1}$. Then, by Corollary \ref{ode} we have
\[y(t)\leq z(t)\leq z(T)\]
for $0<t<T<\min\{T_0, T^*\}$.

To end the proof of the proposition, it remains to estimate $\norm{V_2}_{H^1}$. Since
\[V_2(u_0)=\p^2_t u_\ep|_{t=0},\]
it is not difficult to show
\[\norm{V_2}_{H^1}\leq C(\norm{u_0}_{H^5}).\]
Therefore, by setting  $T_1=\min\{T_0, 0.9T^*\}$, we have
\[\sup_{0<t\leq T_1}\norm{w}^2_{H^{1}}\leq C .\]
Consequently, Lemma \eqref{v-w} implies
\[\sup_{0<t\leq T_1}\norm{v}^2_{H^3}\leq C.\]

On the other hand, since \[\sup_{0<t\leq T_1}\norm{u_\ep}^2_{H^3}\leq C,\] we apply the $L^2$-estimates to elliptic equation
\[\De u_\ep=\frac{1}{1+\ep^2}(\ep v-u_\ep\times v)-|\n u_\ep|^2u_\ep,\]
to show 
\[\sup_{0<t\leq T_1}\norm{u_\ep}^2_{H^4}\leq C.\]

Once we obtain the above improved estimate of $u_\ep$, then we can get
\[\sup_{0<t\leq T_1}\norm{u_\ep}^2_{H^5}\leq C\]
by using equation \eqref{pra-eq0} and the $L^2$-estimates of elliptic equation again.
\end{proof}

\medskip
With the above proposition \ref{H^5-es} at hand, we are in the position to provide the proof of Theorem \ref{thm1} by taking an argument of convergence.

\begin{proof}[\textbf{The proof of Theorem \ref{thm1}}]
Proposition \ref{H^5-es} tells us that there exists a number $T_1>0$ independent of $\ep$ such that $u_\ep$ has the following uniform estimate with respect to $\ep$
\[\sup_{0<t\leq T_1}\norm{\p^i_t u_\ep}^2_{H^{5-2i}(\Om)}\leq C \]
for $i=0, 1, 2$.

Without loss of generality, we assume that there exists a map in $u\in L^\infty([0,T_1], H^5(\Om))$ such that
\[u_\ep\rightharpoonup u\quad \text{weakly* in}\quad u\in L^\infty([0,T_1], H^5(\Om)),\]
and
\[\frac{\p u_\ep}{\p t}\rightharpoonup \frac{\p u}{\p t}\quad \text{weakly in}\quad L^2([0,T_1], H^3(\Om)).\]
Let $X=H^5(\Om)$, $B=H^4(\Om)$ and $Y=L^2(\Om)$.  Then Lemma \ref{A-S} implies
\[u_\ep\to u \quad \text{strongly in}\quad L^\infty([0,T_1], H^4(\Om)),\]
and hence, we have
\[u_\ep\to u \quad  \text{a.e. (x,t)} \in \Om\times[0,T_1]\]
with $|u|=1$.

On the other hand, since $u_\ep$ is a strong solution to \eqref{pra-eq0}, there holds
\[\int_{0}^{T_1}\int_{\Om}\<\frac{\p u_\ep}{\p t},\phi\>dxdt-\ep\int_{0}^{T_1}\int_{\Om}\<\De u_\ep+|\n u_\ep|^2 u_\ep,\phi\>dxdt=\int_{0}^{T_1}\int_{\Om}\<u_\ep\times \De u_\ep,\phi\>dxdt,\]
for all $\phi\in C^\infty(\bar{\Om}\times[0,T_1])$.

By using the above convergence of $u_\ep$, we can show directly that $u$ is a strong solution to \eqref{S-eq} by letting $\ep\to 0$. Moreover, the lower semi-continuity of weak convergence implies
\[\p^i_tu\in L^\infty([0,T_1], H^{5-2i}(\Om)) ,\]
for $i=0,1,2$.

To complete the proof, we need to verify $\frac{\p u}{\p \nu}|_{\p \Om\times[0,T_0]}=0$ which means $u$ satisfies the Neumann boundary condition. Since there holds true that for any $\xi\in C^{\infty}(\bar{\Om}\times[0,T_1])$
$$\int_{0}^{T_1}\int_{\Om}\<\De u_{\ep},\xi\>dxdt=-\int_{0}^{T_1}\int_{\Om}\<\n u_\ep,\n \xi\>dxdt.$$
Let $\ep\to 0$, we have
$$\int_{0}^{T_1}\int_{\Om}\<\De u,\xi\>dxdt=-\int_{0}^{T_1}\int_{\Om}\<\n u,\n \xi\>dxdt,$$
this means
$$\frac{\p u}{\p \nu}|_{\p \Om\times [0,T_1]}=0.$$
\end{proof}
\medskip
\section{Very regular local solution}\label{s: hig-reg}
In this section, we adopt the method of induction to show the existence of very regular solution to \eqref{S-eq} by proving the following theorem, namely Theorem \ref{thm2}.
\begin{thm}\label{V-sol}
Suppose that $u_0\in H^{2k+1}(\Om,\U^2)$ with $k\geq 2$, which satisfies the $(k-1)$-order compatibility conditions defined by \eqref{s-com-cond}. Let $u$ and $T_1>0$ be the same as that in Theorem \ref{thm1}. Then for any $0\leq i\leq k$, we have
\[\p^i_t u\in L^\infty([0,T_1], H^{2k+1-2i}(\Om)).\]
Additionally, if $u_0\in C^\infty(\bar{\Om})$, which satisfies the $k$-order compatibility conditions defined by \eqref{s-com-cond} for any $k\geq 0$, we also have
\[u\in C^\infty(\bar{\Om}\times [0,T_1]).\]
\end{thm}

Recall that the existence of very regular solution $u_\ep$ to the parabolic perturbed equation \eqref{pra-eq0} of equation \eqref{S-eq} has been shown in \cite{CJ} (also see the authors' work \cite{CW}), for the completeness and convenience we summarize the conclusions in below theorem.

\begin{thm}\label{para-V-sol}
Suppose that $u_0\in H^{2k+1}(\Om,\U^2)$ with $k\geq 2$, which satisfies the $(k-1)$-order compatibility condition defined in \eqref{com-cond1}. Let $u_\ep$ and $T_\ep>0$ be the same as that in Theorem \ref{para-H^5}. Then, for $0 \leq i\leq k$ and $0<T<T_\ep$ there holds true
\[\p^i_tu_\ep\in L^\infty([0,T], H^{2k+1-2i}(\Om))\cap L^2([0,T], H^{2k+2-2i}(\Om)).\]
\end{thm}

First of all, we should mention that the compatibility conditions defined by \eqref{s-com-cond} implies the conditions in \eqref{com-cond1}. This guarantees that the approximate solution $u_\ep$ can  certainly tend to a solution of \ref{S-eq}. In the next context,  we shall get higher order uniform energy estimates of $u_\ep$  and then take $\ep\to 0$ to prove Theorem \ref{V-sol}. To this end, we use the method of induction  on $k$ to show higher order uniform energy estimates of $u_\ep$ by considering the equation of
\[w_k=\p^k_t u_\ep\]
with matching initial-boundary data. Namely, we will prove the following proposition.
\begin{prop}\label{hig-es}
Under the same assumption as in the above theorem \ref{para-V-sol}, for $0\leq i\leq k$ there exists a positive constant $C_k$ depending only on $\norm{u_0}_{H^{2k+1}(\Om)}$ such that
\[\sup_{0<t\leq T_1}\norm{w_i}_{H^{2k+1-2i}(\Om)}\leq C_k(\norm{u_0}_{H^{2k+1}(\Om)}).\]
\end{prop}

In fact, one can see easily that, in the previous subsection, the conclusions in Proposition \ref{hig-es} have been shown for $k=2$.

Next, we will use the method of induction on $k$ to show Proposition \ref{hig-es}. Suppose that the estimates in Proposition \ref{hig-es} are already established for $k-1\geq 2$. Then we want to prove that the conclusions are also true in the case of $k$.

For any $k\geq 3$, Theorem \ref{para-V-sol} shows that $w_k\in L^\infty([0,T], H^1(\Om))\cap L^2([0,T], H^2(\Om))$ and satisfies the following equation
\begin{equation}\label{w_k}
\begin{cases}
\p_t w_k=\ep \De w_k+u_\ep\times \De w_k+K_k(\n w_k)+L_k(w_k)+F_k(u_\ep), &\text{(x,t)}\in\Om\times [0, T_\ep),\\[1ex]
\frac{\p w_k}{\p \nu}=0, &\text{(x,t)}\in\p \Om\times [0, T_\ep),\\[1ex]
w(x,0)=V_k(u_0),  &x\in \Om.
\end{cases}
\end{equation}
Here
\begin{equation*}
\begin{aligned}
K_k(\n w_k)=&2\ep \n w_k\cdot \n u_\ep u_\ep,\\
L_k(w_k)=&\ep|\n u_\ep|^2 w_k+w_k\times \De u_\ep,\\
\end{aligned}
\end{equation*}
and
$$F_k(u_\ep)=\ep\sum_{i+j+l=k,\,0\leq i,j,l<k}\n w_i\#\n w_j\#w_l+\sum_{i+j=k,0\leq i,j<k} C^i_kw_i\times \De w_j,$$
where $V_k(u_0)$ is defined in Section \ref{s: com-cond} and $\#$ denotes the linear contraction.

On the other hand, the  assumption of induction shows that for any $i\in \{0, 1,\cdots, k-1\}$ there exists a constant $C_k(\norm{u_0}_{H^{2k-1}})$, which does not depend on $\ep$, such that
\begin{equation}\label{pre-es}
\sup_{0<t\leq T_1}\norm{w_i}_{H^{2(k-i-1)+1}}\leq C_k(\norm{u_0}_{H^{2k-1}(\Om)}).
\end{equation}

Next, we will adopt a similar procedure with that in Section  \ref{s: H^5} for $w_k=\p^k_tu_\ep$ to get the uniform $H^1$-estimates of $w_k$.

\subsection{Estimates of equivalent norms}\
For later application, we need to establish some lemmas on Sobolev space and the equivalent norms of the energy which we need to estimate. We start with recalling the following lemma, the proof of which can be found in \cite{CJ}.

\begin{lem}\label{alg}
Let $\Om$ be a smooth bounded domain in $\Real^3$, $n\geq 0$ and $m\geq 2$. Suppose $f\in H^n(\Om)$ ( and we also denote $H^0(\Om)=L^2(\Om)$) and $g\in H^m(\Om)$, then $fg\in H^l(\Om)$ with $l=\min\{n,m\}$. Moreover, there exists a constant $C(\norm{f}_{H^n}, \norm{g}_{H^m})$ such that we have
\[\norm{fg}_{H^l(\Om)}\leq C(\norm{f}_{H^n}, \norm{g}_{H^m}).\]	
\end{lem}

\begin{lem}\label{w_{k-1}-w_k}
Assume $u_\ep$ is the solution of \eqref{pra-eq0} given in Theorem \ref{para-V-sol}. Then, there exist constants $C_k$ which are independent of $\ep$ such that
\begin{align}
\norm{w_{k-1}}^2_{H^2(\Om)}\leq &C_k(\norm{u_0}^2_{H^{2(k-1)+1}})+2\int_{\Om}|w_k|^2dx,\\
\norm{w_{k-1}}^2_{H^3(\Om)}\leq &C_k(\norm{u_0}^2_{H^{2(k-1)+1}})(\norm{w_k}^2_{H^1}+1),
\end{align}
for a.e. $t\in[0,T_1]$.
\end{lem}

\begin{proof}
Our proof is divided into two steps.

\medskip
\noindent\emph{Step 1: $H^2$-estimates of  $w_{k-1}$.} \

By using equation \eqref{pra-eq0}, we have
\[\De u_\ep=\frac{1}{1+\ep^2}(\ep w_1-u_\ep\times w_1)-|\n u_\ep|^2u_\ep.\]
A direct calculation shows
\begin{align*}
\De w_i=&\frac{1}{1+\ep^2}(\ep w_{i+1}-u_\ep\times w_{i+1}-w_i\times w_1)-2\n w_i\#\n u_\ep \# u_\ep-|\n u_\ep|^2w_i\\
&-\frac{1}{1+\ep^2}\sum_{l+s=i, 0\leq l,s<i}C^l_iw_l\times w_{s+1}-\sum_{l+s+m=i, 0\leq l,s,m<i}\n w_l\#\n w_s\# w_m,
\end{align*}
where $0\leq i\leq k-1$.  And hence, by taking $i=k-1$, it follows

\begin{equation}\label{es1}
\begin{aligned}
\int_{\Om}|\De w_{k-1}|^2dx\leq &2\int_{\Om}|w_{k}|^2dx+C\int_{\Om}|w_{k-1}|^2|w_1|^2dx+C\int_{\Om}|\n w_{k-1}|^2|\n u_\ep|^2dx\\
&+C\int_{\Om}|\n u_\ep|^4|w_{k-1}|^2dx+C\sum_{l+s=k-1, 0\leq l,s<k-1}\int_{\Om}|w_l|^2|w_{s+1}|^2dx\\
&+C\sum_{l+s+m=k-1,0\leq l,s,m<k-1}\int_{\Om}|\n w_l|^2|\n w_s|^2 |w_m|^2dx\\
=&2\int_{\Om}|w_{k}|^2dx+I_1+I_2+I_3+I_4+I_5.
\end{aligned}
\end{equation}

Next, we estimate the last five terms on the right hand side of the above inequality \eqref{es1} term by term.
\begin{align*}
|I_1|=&C\int_{\Om}|w_{k-1}|^2|w_1|^2dx\\
\leq&C\norm{w_{k-1}}^2_{L^2}\norm{w_1}^2_{H^2}\leq C(\norm{u_0}_{H^{2k-1}}(\Om)),
\end{align*}

\begin{align*}
|I_2|=&C\int_{\Om}|\n w_{k-1}|^2|\n u_\ep|^2dx\\
\leq &C\norm{\n w_{k-1}}^2_{L^2}\norm{\n u_\ep}^2_{H^2}\leq C(\norm{u_0}_{H^{2k-1}}(\Om)),
\end{align*}

\begin{align*}
|I_3|=&C\int_{\Om}|\n u_\ep|^4|w_{k-1}|^2dx\\
\leq &C\norm{w_{k-1}}^2_{L^2}\norm{\n u_\ep}^4_{H^2}\leq C(\norm{u_0}_{H^{2k-1}}(\Om)),
\end{align*}

\begin{align*}
|I_4|=&\sum_{l+s=k-1, 0\leq l,s<k-1}\int_{\Om}|w_l|^2|w_{s+1}|^2dx\\
\leq&\sum_{l+s=k-1, 0\leq l,s<k-1}\norm{w_l}^2_{H^1}\norm{w_{s+1}}^2_{H^1}\leq C(\norm{u_0}_{H^{2k-1}}(\Om)),
\end{align*}
and
\begin{align*}
|I_5|=&C\sum_{l+s+m=k-1, 0\leq l,s,m<k-1}\int_{\Om}|\n w_l|^2|\n w_s|^2 |w_m|^2dx\\
\leq &C\sum_{l+s+m=k-1, 0\leq l,s,m<k-1} \norm{w_l}^2_{H^2}\norm{w_s}^2_{H^2}\norm{w_m}^2_{H^1}\\
\leq &C(\norm{u_0}_{H^{2k-1}}(\Om)).
\end{align*}
Here we have used the estimates \eqref{pre-es} obtained by the assumption of induction.

Therefore, plugging the above inequalities $I_1$-$I_5$ into inequality \eqref{es1}, we get the estimate $(5.3)$ by applying Lemma \ref{eq-norm}, since $\frac{\p w_{k-1}}{\p \nu}|_{\p\Om\times [0,T_\ep)}=0$ which are implied by Proposition \ref{comp-cond2}.

\medskip
\noindent\emph{Step 2: $H^3$-estimates of $w_{k-1}$.}\

On the other hand, a simple calculation shows
\begin{equation}\label{nDew}
\begin{aligned}
\int_{\Om}|\n\De w_{k-1}|^2dx\leq &C\int_{\Om}|\n w_{k}|^2dx+C\int_{\Om}|w_{k}|^2|\n u_\ep|^2dx+C\int_{\Om}|\n w_{k-1}|^2|w_1|^2dx\\
	&+C\int_{\Om}|w_{k-1}|^2|\n w_1|^2dx+C\int_{\Om}|\n^2 w_{k-1}|^2|\n u_\ep|^2dx\\
	&+C\int_{\Om}|\n w_{k-1}|^2|\n^2 u_\ep|^2dx+C\int_{\Om}|\n u_\ep|^4|\n w_{k-1}|^2dx\\
	&+C\int_{\Om}|\n^2 u_\ep|^2|\n u_\ep|^2|w_{k-1}|^2dx\\
	&+C\sum_{l+s=k-1, 0\leq l,s<k-1}\int_{\Om}|\n(w_l\#w_{s+1})|^2dx\\
	&+C\sum_{l+s+m=k-1, 0\leq l,s,m<k-1}\int_{\Om}|\n(\n w_l\#\n w_s\#w_m)|^2dx\\
	=&C\int_{\Om}|\n w_{k}|^2dx+M_1+M_2+M_3\\
	&+M_4+M_5+M_6+M_7+M_8+M_9.
\end{aligned}	
\end{equation}
Here, by applying again the estimates \eqref{pre-es} we can show
\begin{align*}
M_1= &C\int_{\Om}|w_{k}|^2|\n u_\ep|^2dx\leq C\norm{u_{\ep}}^2_{H^3}\int_{\Om}|w_{k}|^2dx\\
\leq &C(\norm{u_0}_{H^{2k-1}}(\Om))\int_{\Om}|w_{k}|^2dx,\\
M_2= &C\int_{\Om}|\n w_{k-1}|^2|w_1|^2dx\leq \norm{w_{k-1}}^2_{H^1}\norm{w_1}^2_{H^2}\\
\leq &C(\norm{u_0}_{H^{2k-1}}(\Om)),
\end{align*}
\begin{align*}
M_3=& C\int_{\Om}|w_{k-1}|^2|\n w_1|^2dx\leq \norm{w_{k-1}}^2_{H^1}\norm{w_1}^2_{H^2}\\
&\leq C(\norm{u_0}_{H^{2k-1}}(\Om)),\\
M_4=&C\int_{\Om}|\n^2 w_{k-1}|^2|\n u_\ep|^2dx\leq C\norm{u_{\ep}}^2_{H^3}\int_{\Om}|\n^2 w_{k-1}|^2dx,\\
\leq &C(\norm{u_0}_{H^{2k-1}}(\Om))(1+\int_{\Om}|w_k|^2dx),\\
M_5=&C\int_{\Om}|\n w_{k-1}|^2|\n^2 u_\ep|^2dx\leq C\norm{u_{\ep}}^2_{H^4}\int_{\Om}|\n w_{k-1}|^2dx\\
\leq& C(\norm{u_0}_{H^{2k-1}}(\Om)),\\
M_6=&C\int_{\Om}|\n u_\ep|^4|\n w_{k-1}|^2dx\leq C\norm{u_{\ep}}^4_{H^3}\int_{\Om}|\n w_{k-1}|^2dx\\
&\leq C(\norm{u_0}_{H^{2k-1}}(\Om)),\\
M_7=&C\int_{\Om}|\n^2 u_\ep|^2|\n u_\ep|^2|w_{k-1}|^2dx\leq C\norm{u_{\ep}}^2_{H^3}\norm{u_{\ep}}^2_{H^4}\int_{\Om}|w_{k-1}|^2dx\\
\leq &C(\norm{u_0}_{H^{2k-1}}(\Om)),
\end{align*}
and
 \begin{align*}
M_8\leq&C\sum_{l+s=k-1, 0\leq l,s<k-1}\int_{\Om}(|\n w_l|^2|w_{s+1}|^2+|w_l|^2|\n w_{s+1}|^2)dx\\
\leq & C\sum_{l+s=k-1, 0\leq l,s<k-1}\norm{\n w_l}^2_{H^1}\norm{w_{s+1}}^2_{H^1}\\
&+C\sum_{l+s=k-1, 0\leq l,s<k-1}\norm{w_l}^2_{H^2}\norm{\n w_{s+1}}^2_{L^2}\\
\leq &C(\norm{u_0}_{H^{2k-1}}(\Om)),
\end{align*}
since $2(k-l-1)+1\geq 3$ and $2(k-(s+1)-1)+1\geq 1$ for $l,s<k-1$. Similarly, we can also show
\[M_9\leq C(\norm{u_0}_{H^{2k-1}}(\Om)).\]
Hence, by substituting the estimates on $M_1$-$M_9$ into the above inequality \eqref{nDew} we can obtain
\begin{align*}
\int_{\Om}|\n\De w_{k-1}|^2dx\leq &C(\norm{u_0}_{H^{2k-1}}(\Om))(\norm{w_k}^2_{H^1}+1).
\end{align*}
Therefore, we can use Lemma \ref{eq-norm} to get the desired result $(5.4)$.
\end{proof}

With Lemma \ref{w_{k-1}-w_k} at hand, it is not difficult to show the following estimate of the nonhomogeneous term $F_k$ in equation \eqref{w_k}.
\begin{lem}\label{es-error}
Under the assumption of induction (namely Proposition \ref{hig-es} holds for $k-1\geq 2$), there exists a constant $C_k$, which is independent of $\ep$, such that for a.e. $t\in[0,T_1]$
\[\int_{\Om}|F_k|^2dx\leq C_k(1+\int_{\Om}|w_k|^2dx).\]
\end{lem}
\begin{proof}
A direct computation shows
\begin{align*}
F_k(u_\ep)=&\ep\sum_{i+j+l=k,\,0\leq i,j,l<k}\n w_i\#\n w_j\#w_l+\sum_{i+j=k,0\leq i,j<k} C^i_kw_i\times \De w_j\\
=&\ep\n w_{k-1}\#\n w_1\#u_\ep+\ep\n w_{k-1}\#\n u_\ep\#w_1+\ep\n w_1\#\n u_\ep\#w_{k-1}\\
&+kw_{k-1}\times \De w_1+kw_1\times \De w_{k-1}\\
&+\ep\sum_{i+j+l=k,\,0\leq i,j,l<k-1}\n w_i\#\n w_j\#w_l+\sum_{i+j=k,0\leq i,j<k-1} C^i_kw_i\times \De w_j\\
=&\ep\n w_{k-1}\#\n w_1\#u_\ep+\ep\n w_{k-1}\#\n u_\ep\#w_1+\ep\n w_1\#\n u_\ep\#w_{k-1}\\
&+kw_{k-1}\times \De w_1+kw_1\times \De w_{k-1}+\tilde{F}_k(u_\ep)\\
=&I_1^* + I_2^* + I_3^* + I_4^* + I_5^* + \tilde{F}_k(u_\ep).
\end{align*}
Here, for the sake of simplicity we denote
\[\tilde{F}_k(u_\ep)=\ep\sum_{i+j+l=k,\,0\leq i,j,l<k-1}\n w_i\#\n w_j\#w_l+\sum_{i+j=k,0\leq i,j<k-1} C^i_kw_i\times \De w_j.\]

Now, by using the estimates \eqref{pre-es} (the assumption of induction) and H\"older inequality, we can estimate the six terms on the right hand side of the above identity as follows.
\begin{align*}
\int_{\Om}|I_1^*|^2dx\leq &C\ep\norm{w_{k-1}}^2_{H^1}\norm{w_1}^2_{H^3}\leq C(\norm{u_0}_{H^{2k-1}}),\\
\int_{\Om}|I_2^*|^2dx\leq &C\ep\norm{w_{k-1}}^2_{H^1}\norm{u_\ep}^2_{H^3}\norm{w_1}^2_{H^2}\leq C(\norm{u_0}_{H^{2k-1}}),\\
\int_{\Om}|I_3^*|^2dx\leq &C\ep\norm{w_{k-1}}^2_{H^1}\norm{u_\ep}^2_{H^3}\norm{w_1}^2_{H^2}\leq C(\norm{u_0}_{H^{2k-1}}),\\
\int_{\Om}|I_4^*|^2dx\leq &C\norm{w_{k-1}}^2_{H^1}\norm{w_1}^2_{H^3}\leq C(\norm{u_0}_{H^{2k-1}}),\\
\int_{\Om}|I_5^*|^2dx\leq &C\norm{w_{k-1}}^2_{H^2}\norm{w_1}^2_{H^2}\leq C(\norm{u_0}_{H^{2k-1}})(1+\int_{\Om}|w_k|^2dx).
\end{align*}

It remains to estimate the $L^2$-norm of $\tilde{F}(u_\ep)$. For $i,j,l<k-1$, by using Lemma \ref{alg}, we have
\[\norm{\n w_i\#\n w_j\# w_l}_{H^{2}(\Om)}\leq C(\norm{u_0}_{H^{2k-1}}(\Om)).\]
On the other hand, we have
\[\norm{\De w_j}_{H^{1}(\Om)}\leq C\norm{w_j}_{H^3}\leq C(\norm{u_0}_{H^{2k-1}}(\Om))\]
for $j<k-1$. Then, using again Lemma \ref{alg} leads to
\[\norm{w_i\#\De w_j}_{H^{1}(\Om)}\leq C(\norm{u_0}_{H^{2k-1}}(\Om))\]
for any $i,j<k-1$. Namely, there holds
	\[\norm{\tilde{F}_k(u_\ep)}_{H^{1}(\Om)}\leq C(\norm{u_0}_{H^{2k-1}}(\Om)).\]
Therefore, we can easily get the desired estimates from the estimates on $I_1^*$-$I_5^*$ and $\tilde{F}(u_\ep)$.
\end{proof}

\subsection{Uniform $L^2$-estimate of $w_k$.}\

Now, we intend to show a uniform $L^2$-estimate of $w_k$ by direct energy estimates. By taking $w_k$ as a test function of \eqref{w_k}, we have
\begin{equation}\label{L^2-w_k}
\begin{aligned}
\frac{1}{2}\frac{\p}{\p t}\int_{\Om}|w_k|^2dx+\ep\int_{\Om}|\De w_k|^2dx=&\int_{\Om}\<u_\ep\times \De w_k, w_k\>dx+\int_{\Om}\<K_k(\n w_k), w_k\>dx\\
&+\int_{\Om}\<L_k(w_k), w_k\>dx+\int_{\Om}\<F_k(u_\ep), w_k\>dx\\
=& J_1 + J_2 + J_3 + J_4.
\end{aligned}	
\end{equation}
Now we give the estimates of the four terms on the right hand side of the above inequality \ref{L^2-w_k} respectively as follows.
\begin{align*}
|J_1|\leq & \left|\int_{\Om}\<u_\ep\times \De w_k, w_k\>dx\right|\leq C\int_{\Om}|\n u_\ep||\n w_k||w_k|dx,\\
\leq &C\norm{u_{\ep}}_{H^3}\int_{\Om}|\n w_k|^2+|w_k|^2dx,\\
|J_2|\leq& 2\ep\left|\int_{\Om}\<\n w_k\cdot \n u_\ep u_\ep, w_k\>dx\right|\\
\leq &C\ep \norm{u_\ep}^2_{H^3}\int_{\Om}|w_k|^2dx+\frac{\ep}{2}\int_{\Om}|\n w_k|^2dx,\\
|J_3|=&\left|\int_{\Om}\<L_k(w_k), w_k\>dx\right|\leq \ep \int_{\Om}|\n u_\ep|^2|w_k|^2dx\\
\leq &C\ep \norm{u_\ep}^2_{H^3}\int_{\Om}|\n w_k|^2dx,\\
|J_4|\leq& C\int_{\Om}|F_k(u_\ep)||w_k|dx\\
\leq &C\int_{\Om}|F_k(u_\ep)|^2dx+C\int_{\Om}|w_k|^2dx\\
\leq &C(\norm{u_0}_{H^{2k-1}}(\Om))+C\int_{\Om}|w_k|^2dx.
\end{align*}
Therefore, by substituting the above estimates $J_1$-$J_4$ into the inequality \eqref{L^2-w_k}, we have
\begin{align}
\frac{\p}{\p t}\int_{\Om}|w_k|^2dx+\ep\int_{\Om}|\De w_k|^2dx\leq C(\norm{u_0}_{H^{2k-1}}(\Om))\(1+\int_{\Om}(|w_k|^2+|\n w_k|^2)dx\),
\end{align}
where the constant $C$ does not depend on $\ep$.

\medskip
\subsection{Uniform $H^1$-estimate of $w_k$.}\
To get a uniform bound of $H^1$-norm of $w_k$ with respect to $\ep$, we should enhance the regularity of $w_k$ to guarantee that integration by parts makes sense during the process of energy estimates. By Theorem \ref{para-V-sol}, we know
$$w_k\in L^\infty([0,T], H^{1}(\Om))\cap L^2([0,T], H^{2}(\Om)),$$
which satisfies the following equation
\begin{equation}
\begin{cases}
\p_t w_k =\ep \De w_k+u_\ep\times \De w_k+f_k,\\[1ex]
\frac{\p w_k}{\p \nu}|_{\p\Om\times [0,T_\ep)}=0, \\[1ex]
w(x,0)=V_k,
\end{cases}
\end{equation}
where
\[f_k=K_k(\n w_k)+L_k(w_k)+F_k(u_\ep).\]

Since we have shown
\[F_{k}(u_\ep)\in L^2([0,T], H^2(\Om))\]
in Proposition 4.4 of \cite{CW}, it is not difficult to get
\[f_k\in L^2([0,T], H^1(\Om)),\]
for any $0<T<T_\ep$.

Hence, the classical $L^2$-estimates of parabolic equation (also see Theorem A.1 in \cite{CW1}) tells us that
\[w_k\in L^2_{loc}((0,T], H^{3}(\Om))\]
and
\[\frac{\p w_k}{\p t}\in L^2_{loc}((0,T], H^{1}(\Om)),\]
which guarantee integration by parts in the following process of energy estimates makes sense. By choosing $\De w_k$ as a test function of \eqref{w_k}, we have
\begin{equation}\label{H^1}
\begin{aligned}
\frac{1}{2}\frac{\p}{\p t}\int_{\Om}|\n w_k|^2dx+\ep\int_{\Om}|\De w_k|^2dx=&-\int_{\Om}\<K_k(\n w_k), \De w_k\>dx-\int_{\Om}\<L_k(w_k), \De w_k\>dx
\\&-\int_{\Om}\<F_k(u_\ep), \De w_k\>dx\\
=& J_1^*+ J_2^* + J_3^* + J_4^*.
\end{aligned}
\end{equation}
Now, we estimate the four terms on the right hand side of the above inequality \ref{H^1} term by term. Firstly, we have
\begin{align*}
|J_1^*|=&\left|\int_{\Om}\<K_k(\n w_k), \De w_k\>dx\right|\\
\leq&2\ep\left|\int_{\Om}\<\n w_k\cdot \n u_\ep u_\ep, \De w_k\>dx\right|\\
\leq& C\ep\norm{u_\ep}^2_{H^3}\int_{\Om}|\n w_k|^2dx+\frac{\ep}{8}\int_{\Om}|\De w_k|^2dx,\\
|J_2^*|=&\left|\int_{\Om}\<L_k(w_k), \De w_k\>dx\right|\\
\leq&\ep\left|\int_{\Om}\<|\n u_\ep|^2 w_k, \De w_k\>dx\right| + \left|\int_{\Om}\<w_k\times \n \De u_\ep, \n w_k\>dx\right|\\
\leq&C\ep\norm{u_\ep}^4_{H^3}\int_{\Om}|w_k|^2dx+\frac{\ep}{8}\int_{\Om}|\De w_k|^2dx\\
&+\norm{u_\ep}^2_{H^5}\int_{\Om}|w_k|^2dx+C\int_{\Om}|\n w_k|^2dx.
\end{align*}
Applying similar arguments as that in the proof of Lemma \ref{es-error} leads to
\begin{align*}
|J_3^*|=&\ep \left|\int_{\Om}\<\sum_{i+j+l=k,\,0\leq i,j,l<k}\n w_i\#\n w_j\#w_l, \De w_k\>dx\right|\\
\leq& C\ep\sum_{i+j+l=k,\,0\leq i,j,l<k}\int_{\Om}|\n w_i|^2|\n w_j|^2|w_l|^2dx+\frac{\ep}{8}\int_{\Om}|\De w_k|^2dx\\
\leq& \ep C(\norm{u_0}_{H^{2k-1}}(\Om))+\frac{\ep}{8}\int_{\Om}|\De w_k|^2dx.
\end{align*}
 For the last term $J^*_4$, we have
\begin{align*}
|J_4^*|=&C\left|\int_{\Om}\sum_{i+j=k,0\leq i,j<k} \<w_i\times \De w_j, \De w_k\>dx\right|\\
\leq& C\sum_{i+j=k,0\leq i,j<k}|\int_{\Om}\<\n w_i\times \De w_j, \n w_k\>dx|\\
&+C\sum_{i+j=k,0\leq i,j<k}|\int_{\Om}\<w_i\times \n \De w_j, \n w_k\>dx|\\
=&C(a^*+b^*).
\end{align*}
Here,
\begin{align*}
a^*=&\left|\int_{\Om}\<\n w_{k-1}\times \De w_1, \n w_k\>dx\right| + \left|\int_{\Om}\<\n w_1\times \De w_{k-1}, \n w_k\>dx\right|\\
&+\sum_{i+j=k,0\leq i,j<k-1}\left|\int_{\Om}\<\n w_i\times \De w_j, \n w_k\>dx\right|\\
\leq&\norm{\n w_{k-1}}_{L^6}\norm{\De w_1}_{L^3}\norm{\n w_k}_{L^2}+\norm{\n w_1}_{L^\infty}\norm{\De w_{k-1}}_{L^2}\norm{\n w_{k}}_{L^2}\\
&+\sum_{i+j=k,0\leq i,j<k-1}\norm{\n w_i}_{L^6}\norm{\De w_j}_{L^3}\norm{\n w_{k}}_{L^2}\\
\leq& C\norm{w_1}^2_{H^3}\norm{w_{k-1}}^2_{H^2}+C\norm{\n w_k}^2_{L^2}+C\sum_{i+j=k,0\leq i,j<k-1}\norm{w_i}^2_{H^2}\norm{w_j}^2_{H^3}\\
\leq& C_k(1+\norm{w_k}^2_{H^1}),
\end{align*}
and
\begin{align*}
b^*=&\left|\int_{\Om}\<w_{k-1}\times \n \De w_1, \n w_k\>dx\right| + \left|\int_{\Om}\<w_1\times \n \De w_{k-1}, \n w_k\>dx\right|\\
&+\sum_{i+j=k,0\leq i,j<k-1}\left|\int_{\Om}\<w_i\times \n \De w_j, \n w_k\>dx\right|\\
\leq& \norm{w_{k-1}}_{L^\infty}\norm{\n \De w_1}_{L^2}\norm{\n w_k}_{L^2}+\norm{w_1}_{L^\infty}\norm{\n \De w_{k-1}}_{L^2}\norm{\n w_k}_{L^2}\\
&+\sum_{i+j=k,0\leq i,j<k-1}\norm{w_i}_{L^\infty}\norm{\n \De w_j}_{L^2}\norm{\n w_{k}}_{L^2}\\
\leq&\norm{w_{k-1}}^2_{H^2}\norm{w_1}^2_{H^3}+C\norm{\n w_k}^2_{L^2}+C\norm{w_1}^2_{H^2}\norm{\n \De w_{k-1}}^2_{L^2}\\
&+C\sum_{i+j=k,0\leq i,j<k-1}\norm{w_i}^2_{H^2}\norm{w_j}^2_{H^3}\leq C_k(1+\norm{w_k}^2_{H^1}),
\end{align*}
where we have used Lemma \ref{w_{k-1}-w_k} and the estimates \eqref{pre-es} from the assumption of induction. Hence, it follows that
\[\sup_{0<t\leq T_1}\norm{w_i}_{H^3}\leq C_k.\]
since $2(k-i-1)+1\geq 3$ for $0\leq i<k-1$.

Therefore, by combining the above estimates with formula \eqref{H^1}, we have
\begin{equation}\label{H^1-w_k}
\frac{\p}{\p t}\int_{\Om}|\n w_k|^2dx+\ep\int_{\Om}|\De w_k|^2dx\leq C(\norm{u_0}_{H^{2k-1}(\Om)})(1+\norm{w_k}^2_{H^1})
\end{equation}
where $C$ does not depend on $\ep\in(0,\,1)$.

To end this section, under the assumption of induction (i.e. the estimates \eqref{pre-es} hold), we combine the inequalities \eqref{L^2-w_k} with \eqref{H^1-w_k} to show the conclusions of Proposition \ref{hig-es} are also true in the case of $k\geq 3$.

\begin{prop}\label{H-es}
Let $u_\ep$ and $T_\ep>0$ be the same as those given in Theorem \ref{para-H^5} and $0<T_1<T_\ep$ be the positive time obtained in Proposition \ref{H^5-es}. Assume that $u_\ep$ satisfies the estimates \eqref{pre-es} for any $0\leq i\leq k-1$, where $k\geq 3$. If $u_0\in H^{2k+1}(\Om,\U^2)$ and satisfies the $(k-1)$-order compatibility condition defined in \eqref{com-cond}, then there exists a constant $C_k$ independent of $\ep$ such that for any $0\leq i\leq k$, there holds true
\[\sup_{0<t\leq T_1}\norm{\p^i_t u_\ep}_{H^{2(k-i)+1}}\leq C_k(\norm{u_0}_{H^{2k+1}(\Om)}).\]
\end{prop}
\begin{proof}
Our proof is divided into three steps.

\medskip
\noindent\emph{Step 1: Estimates of $w_k$.} \

By combining inequalities \eqref{L^2-w_k} with \eqref{H^1-w_k}, we can show
\[\frac{\p}{\p t}\int_{\Om}(|w_k|^2+|\n w_k|^2)dx\leq C_k(1+\int_{\Om}(|w_k|^2+|\n w_k|^2)dx),\]
for $0\leq t\leq T_1$.
Then, Gronwall inequality implies
\[\sup_{0<t\leq T_1}\int_{\Om}(|w_k|^2+|\n w_k|^2)dx\leq e^{C_kT_1}(\norm{V_k(u_0)}^2_{H^1}+1).\]
It is not difficult to verify
\[\norm{V_k(u_0)}^2_{H^1}\leq C(\norm{u_0}_{H^{2k+1}}).\]

\noindent\emph{Step 2: Estimates of $w_i$ for $1\leq i<k$.} \

We show the estimates of $w_i$ with $0\leq i\leq k$ by applying the method of induction on $n=k-i$. Since the desired estimates of $w_k$ have been obtained in above, that is the case of $n=0$, we assume that the result has been established for $n\leq j$ where $j\leq k-2$. Then, in the case that $n=j+1$, a simple calculation shows
\begin{align*}
\De w_{k-j-1}=&\frac{1}{1+\ep^2}(\ep w_{k-j}-u_\ep\times w_{k-j}-w_{k-j-1}\times w_1)\\
&-2\n w_{k-j-1}\#\n u_\ep\# u_\ep-|\n u_\ep|^2w_{k-j-1}\\
&-\frac{1}{1+\ep^2}\sum_{l+s=k-j-1, 0\leq l,s<k-j-1}w_l\times w_{s+1}\\
&-\sum_{l+s+m=k-j-1, 0\leq l,s,m<k-j-1}\n w_l\#\n w_s\# w_m\\
&=K_1 + K_2 + K_3.
\end{align*}
Here,
\begin{align*}
	K_1=&\frac{1}{1+\ep^2}\(\ep w_{k-j}-u_\ep\times w_{k-j}-w_{k-j-1}\times w_1\)\\
	&-2\n w_{k-j-1}\#\n u_\ep\# u_\ep-|\n u_\ep|^2w_{k-j-1}.
\end{align*}

Next we estimate the three terms $K_1$, $K_2$ and $K_3$ in the above respectively.
\begin{itemize}
	\item[$(1)$]  For the term $K_1$,
since $2(k-(k-j-1)-1)+1=2j+1$ with $0\leq j\leq k-2$ and $k\geq 3$, by using the estimates \eqref{pre-es}, we have
	\[\sup_{0<t\leq T_1}(\norm{w_{k-j-1}}_{H^{2j+1}}+\norm{u_\ep}_{H^{2k-1}}+\norm{w_{1}}_{H^{2(k-2)+1}})\leq C_k.\]
On the other hand, by using the assumption of induction, we know that there holds true
	\[\sup_{0<t\leq T_1}\norm{w_{k-j}}_{H^{2j+1}}\leq  C_k.\]
	Therefore, Lemma \ref{alg} implies
	\[\sup_{0<t\leq T_1}\norm{K_1}_{H^{2j}}\leq C_k.\]
	\item[$(2)$] For the term $$K_2=\frac{1}{1+\ep^2}\sum_{l+s=k-j-1, 0\leq l,s<k-j-1}w_l\times w_{s+1},$$
since there holds
	\[\sup_{0<t\leq T_1}(\norm{w_l}_{H^{2(j+1)+1}}+\norm{w_{s+1}}_{H^{2j+1}})\leq C_k\]
   for $s,l\leq k-j-2$, Lemma \ref{alg} tells us that
   \[\sup_{0<t\leq T_1}\norm{K_2}_{H^{2j+1}}\leq C_k.\]

   \item[$(3)$] For the third term $$K_3=\sum_{l+s+m=k-j-1, 0\leq l,s,m<k-j-1}\n w_l\#\n w_s\# w_m,$$
  by a similar argument with that for the term $K_2$, we can use Lemma \ref{alg} again to show
   \[\sup_{0<t\leq T_1}\norm{K_3}_{H^{2(j+1)}}\leq C_k.\]	
\end{itemize}

Therefore, we has obtained
\[\sup_{0<t\leq T_1}\norm{\De w_{k-j-1}}_{H^{2j}}\leq C_k.\]
It follows the above estimates, the classical $L^2$-estimates and Lemma \ref{eq-norm}
\[\sup_{0<t\leq T_1}\norm{w_{k-j-1}}_{H^{2(j+1)}}\leq C_k.\]

Once the regularity of $w_{k-j-1}$ is improved, there is an improved bound of $K_1$:
\[\sup_{0<t\leq T_1}\norm{K_1}_{H^{2j+1}}\leq C_k.\]
And hence, it follows
\[\sup_{0<t\leq T_1}\norm{\De w_{k-j-1}}_{H^{2j+1}}\leq C_k.\]
Again the classical $L^2$-estimates  and Lemma \ref{eq-norm} implies
\[\sup_{0<t\leq T_1}\norm{w_{k-j-1}}_{H^{2(j+1)+1}}\leq C_k.\]

\medskip
\noindent\emph{Step 3: Estimates of $u_\ep$.} \

In the above step $2$, we have gotten a bound of $w_1$ stated as follows
 \[\sup_{0<t\leq T_1}\norm{w_1}_{H^{2k-1}}\leq C_k.\]
On the other hand, we have
 \[\De u_\ep=\frac{1}{1+\ep^2}(\ep w_1-u_\ep\times w_1)-|\n u_\ep|^2u_\ep.\]
Since $\norm{u_\ep}_{H^{2k-1}}\leq C_k$, by the above equation and Lemma \ref{alg}, it is not difficult to verify the following
 \[\sup_{0<t\leq T_1}\norm{\De u_\ep}_{H^{2k-2}}\leq C_k.\]
Immediately, it follows from the classical $L^2$-estimates theory that
 \[\sup_{0<t\leq T_1}\norm{u_\ep}_{H^{2k}}\leq C_k.\]
Hence, by using $L^2$-theory again we can improve the estimate of $\De u_\ep$ to achieve
 \[\sup_{0<t\leq T_1}\norm{\De u_\ep}_{H^{2k-1}}\leq C_k.\]
This leads to
 \[\sup_{0<t\leq T_1}\norm{u_\ep}_{H^{2k+1}}\leq C_k.\]
\end{proof}

\subsection{The proof of Theorem \ref{V-sol}.}\
In this subsection, we prove Theorem \ref{V-sol}.

\begin{proof}[The proof of Theorem \ref{V-sol}]
Suppose that $u_0\in H^{2k+1}(\Om,\U^2)$ with $k\geq 2$, which satisfies the $(k-1)$-order compatibility condition defined in \eqref{com-cond1}. For any $0\leq i\leq k$, Proposition \ref{hig-es} tells us that the following uniform estimates of $u_\ep$ hold
\[\sup_{0<t\leq T_1}\norm{\p^i_t u_\ep}_{H^{2(k-i)+1}}\leq C_k.\]
Hence, an argument on convergence shows that there exists a limiting map $u\in L^\infty([0,T], H^{2k+1}(\Om))$ solving \eqref{S-eq}. Moreover, the lower semicontinuity of weak convergence implies that $u$ also satisfies
\begin{align}
\sup_{0<t\leq T_1}\norm{\p^i_t u}_{H^{2(k-i)+1}}\leq C_k\label{es-u-k}	
\end{align}
for any $0\leq i\leq k$.

Additionally, if $u_0\in C^\infty(\bar{\Om})$, which satisfies the $k$-order compatibility conditions defined by \eqref{s-com-cond} for any $k\geq 0$, the above estimates \eqref{es-u-k} yield that
\[\sup_{0<t<T_1}\norm{\p^j_t\p^s_xu}^2_{L^2}<\infty\]
for any $j,s\in \mathbb{N}$. So, it follows from the Sobolev embedding theorem that
\[u\in C^\infty(\bar{\Om}\times[0,T_1]),\] 
 
Therefore, the proof is completed.

\end{proof}

\section{Global existence of smooth solutions to 1-dimensional Schr\"odinger flow}\label{s: global-exist}
In this section, we are concerned with the global existence of regular solutions to the following initial-Neumann boundary value problem of the 1-dimensional Schr\"odinger flow
\begin{equation}\label{eq-ISMF}
\begin{cases}
\p_tu=u\times \p^2_xu,\quad\quad&\text{(x,t)}\in(0,1)\times \Real^+,\\[1ex]
\p_xu(0,t)=0,\,\p_xu(1,t)=0, &t\in\Real^+,\\[1ex]
u(x,0)=u_0: \Om\to \U^2,
\end{cases}
\end{equation}
where $u$ is a time-dependent map from $(0,1)$ into a standard sphere $\U^2$. For simplicity, we set $I=[0,1]$.

Recall that Theorem \ref{thm2} implies the following result about the local existence of smooth solution to \eqref{eq-ISMF}.
\begin{thm}\label{thm4}
Suppose that $u_0\in C^\infty (I,\U^2)$, which satisfies the $k$-order compatibility condition defined in \eqref{com-cond3} for any $k\in \mathbb{N}$. Then there exists a positive maximal time $T_{max}$ depending only on $\norm{u_0}_{H^5(I)}$ such that the initial-Neumann boundary value problem \eqref{eq-ISMF} admits a unique local smooth solution $u$ on $[0,T_{max})$.
\end{thm}
\begin{proof}
By applying Theorem \ref{thm2}, we know that there exists a positive time $T_1$ depending only on $\norm{u_0}_{H^5(I)}$ such that the problem \eqref{eq-ISMF} admits a smooth local solution $u\in C^\infty(I\times [0,T_1])$.

On the other hand, it is not difficult to show that $u(x,T_1)$ meets the same compatibility condition as $u_0$ defined in \eqref{com-cond3}. Then, $T_1$ is an extendable time. Therefore, by taking the same arguments as in the proof of Theorem \ref{thm2} we can get a maximal existence time $T_{max}$ depending only on $\norm{u_0}_{H^5(I)}$ such that equation \eqref{eq-ISMF} admits a smooth solution $u$ on $[0,T_{max})$.
\end{proof}

Next, we show energy estimates for the local solution $u$. Without lose of generality, we use $C$ to denote constants independent of $u$ and $T$ appearing in the estimates in the following context. And especially, for any $k\in \mathbb{N}$, we use $C(\norm{u_0}^2_{H^{2k+1}}, T)$ to denote the constants depending only on $\norm{u_0}^2_{H^{2k+1}}$ and $T$, such that
\[C(\norm{u_0}^2_{H^{2k+1}}, T)<\infty\]
if $T<\infty$. For simplicity, we also denote the partial derivatives of any vector valued function $f$ by $f_t=\p_tf$ and $f_x=\p_xf$.

\subsection{$H^2$-energy estimate}
For any $T<T_{max}$, a simple calculation shows
\begin{equation}\label{ineq1}
\frac{\p}{\p t}\int_{I}|u_x|^2dx=2\int_{I}\<u_x,u_{xt}\>dx=-2\int_{I}\<u_{xx},u_{t}\>dx=0,
\end{equation}
and
\begin{align*}
\frac{\p}{\p t}\int_{I}|u_t|^2dx=&2\int_{I}\<u_t, u_{tt}\>dx=2\int_{I}\<u_t, (u\times u_{xx})_t\>dx\\
=&2\int_{I}\<u_t, u_t\times u_{xx}\>dx+2\int_{I}\<u_t, u\times u_{xxt}\>dx\\
=&-2\int_{I}\<u\times u_t, u_{txx}\>dx\\
=&2\int_{I}\<u_x\times u_t, u_{tx}\>dx+2\int_{I}\<u\times u_{tx}, u_{tx}\>dx\\
=&2\int_{I}\<u_x\times (u\times u_{xx}), u_{tx}\>dx\\
=&2\int_{I}\<u_x,u_{xx}\>\<u, u_{tx}\>dx-2\int_{I}\<u_x,u\>\<u_{xx}, u_{tx}\>dx\\
=&-\int_{I}|u_x|^2_x\<u_x, u_{t}\>dx=\int_{I}|u_x|^2\<u_x, u_{t}\>_xdx\\
=&\int_{I}|u_x|^2(\<u_{xx}, u_{t}\>+\<u_x,u_{xt}\>)dx\\
=&\int_{I}|u_x|^2\<u_x,u_{xt}\>dx=\frac{1}{4}\frac{\p}{\p t}\int_{I}|u_x|^4dx,
\end{align*}
namely,
\begin{equation}\label{ineq2}
\frac{\p}{\p t}\(\int_{I}|u_t|^2dx-\frac{1}{4}\int_{I}|u_x|^4dx\)=0.
\end{equation}
Here we have applied the facts:
\begin{itemize}
\item[$(1)$] $u_x(0,t)=u_x(1,t)=0$ and $u_{tx}(0,t)=u_{tx}(1,t)=0$ for any $t\in [0,T)$,
\item[$(2)$] $\<u_{xx}, u_t\>=0$ and $\<u, u_x\>=\<u,u_t\>=0$.
\end{itemize}
Since $|u_t|^2|_{t=0}=|\tau(u_0)|$ with $\tau(u_0)=u_{0xx}+|u_{0x}|^2u_0$, we have
\begin{align*}
\int_{I}|u_t|^2dx(t)=&\frac{1}{4}\int_{I}|u_x|^4dx(t)+\int_{I}|\tau(u_0)|^2dx-\frac{1}{4}\int_{I}|u_{0x}|^4dx\\
\leq&\frac{1}{4}\int_{I}|u_x|^4dx(t)+\int_{I}|\tau(u_0)|^2dx.
\end{align*}

To proceed, we need to recall the following Sobolev interpolation inequality.
\begin{lem}
Let $1\leq q,r\leq \infty$, $0\leq j\leq k\in \mathbb{N}$. For $f\in C^\infty(\Om)$ with $\text{dim}(\Om)=m$, there holds
\begin{equation}\label{sob-int-ineq}
\norm{\p_x^jf}_{L^p}\leq C_{p,r,q,j,k}\norm{f}^a_{H^k}\norm{f}^{1-a}_{L^q}
\end{equation}
where $p,q,r,a$ satisfies
\[\frac{1}{p}=\frac{j}{m}+a(\frac{1}{r}-\frac{k}{m})+\frac{1-a}{q}\]
and
\[\frac{j}{k}\leq a\leq 1.\]

In the case $\frac{1}{r}=\frac{k-j}{m}\neq 1$, inequality \eqref{sob-int-ineq} is not valid for $a=1$.
\end{lem}

As a corollary, we have
\begin{cor}
For any $u\in C^\infty(I)$, there holds
\[\norm{u_x}_{L^4}\leq C\norm{u_x}^\frac{1}{4}_{H^1}\norm{u_x}^\frac{3}{4}_{L^2}.\]
\end{cor}
\begin{proof}
This inequality is just inequality \eqref{sob-int-ineq} with $m=1,\,j=0,\, k=1,\, q=r=2$ and $f=u_x$.
\end{proof}

On the other hand, since
\[u_{xx}=-u\times u_t-|u_x|^2u,\]
one can show
\begin{align*}
\int_{I}|u_{xx}|^2dx=&\int_{I}|u_t|^2dx+\int_{I}|u_x|^4dx\\
\leq &\frac{5}{4}\int_{I}|u_x|^4dx+\int_{I}|\tau(u_0)|^2dx\\
\leq &C\norm{u_x}_{H^1}\norm{u_x}^3_{L^2}+\int_{I}|\tau(u_0)|^2dx\\
\leq &C\norm{u_x}^4_{L^2}+C\norm{u_{xx}}_{L^2}\norm{u_x}^3_{L^2}+\int_{I}|\tau(u_0)|^2dx\\
\leq &C(\norm{u_x}^2_{L^2}+1)^3+\int_{I}|\tau(u_0)|^2dx+\frac{1}{2}\norm{u_{xx}}^2_{L^2}.
\end{align*}
This implies
\[\int_{I}|u_{xx}|^2dx\leq C(\norm{u_{0x}}^2_{L^2}+1)^3+\int_{I}|\tau(u_0)|^2dx.\]

Hence, we concludes that
\begin{prop}
For any $T<T_{max}$, the solution $u$ satisfies an energy estimate
\begin{equation}\label{ineq3}
\sup_{0<t<T}(\norm{u}^2_{H^2}+\norm{u_t}^2_{L^2})\leq C(\norm{u_{0}}^2_{H^1}+1)^3+\norm{\tau(u_0)}^2_{L^2}.
\end{equation}
\end{prop}

\subsection{$H^3$-energy estimate}
Since $u_t=u\times u_{xx}$ and $u_{xx}=-u\times u_t-|u_x|^2u$, a simple calculation shows
\begin{equation}\label{eq-H3}
u_{tt}=u\times u_{txx}+u_t\times u_{xx}
\end{equation}
and
\begin{equation}\label{1eq-H3}
u_{xxx}=-u_x\times u_t-u\times u_{tx}-2u_{xx}\cdot u_x u-|u_x|^2u_x.
\end{equation}

Taking $-u_{txx}$ as a test function to \eqref{eq-H3}, we can see
\begin{align*}
\frac{1}{2}\frac{\p}{\p t}\int_{I}|u_{tx}|^2=&-\int_I\<u_t\times u_{xx}, u_{txx}\>dx\\
=&\int_I\<u_t\times u_{xxx}, u_{tx}\>dx,
\end{align*}
where we have used the boundary condition $u_{xt}(0,t)=0,\, u_{xt}(1,t)=0$ to cancel the boundary term arising from integration by parts. Then substituting \eqref{1eq-H3} into the above formula, one can show
\begin{align*}
\frac{1}{2}\frac{\p}{\p t}\int_{I}|u_{tx}|^2=&\int_I\<u_t\times u_{xxx}, u_{tx}\>dx\\
=&-\int_I\<u_t\times (u_x\times u_t), u_{tx}\>dx-\int_I\<u_t\times (u\times u_{tx}), u_{tx}\>dx\\
&-2\int_I\<u_t\times u, u_{tx}\>\<u_{xx}, u_x\>dx-\int_I|u_x|^2\<u_t\times u_x, u_{tx}\>dx\\
=&A_1+A_2+A_3+A_4.
\end{align*}
Then we estimate the above four terms $A_1-A_4$ as follows.
\begin{align*}
|A_1|=&\left|\int_I\<u_t\times (u_x\times u_t), u_{tx}\>dx\right|\\
=&\left|\int_I|u_t|^2\<u_x, u_{tx}\>dx-\int_I\<u_t,u_x\>\<u_t, u_{tx}\>dx\right|\\
\leq &C\int_I|u_t|^2|u_x||u_{tx}|dx\leq C|u_t|_{L^\infty}|u_x|_{L^\infty}\norm{u_t}_{L^2}\norm{u_{tx}}_{L^2}\\
\leq &C\norm{u}_{H^2}\norm{u_t}_{H^1}\norm{u_t}_{L^2}\norm{u_{tx}}_{L^2}\\
\leq & C(\norm{u}^2_{H^2}+\norm{u_t}^2_{L^2}+1)^2(\norm{u_{tx}}^2_{L^2}+1),
\end{align*}
where we have used the following Sobolev embedding
\[C^0(I)\hookrightarrow W^{1,1}(I)\hookrightarrow H^1(I).\]
\begin{align*}
|A_2|=&\left|\int_I\<u_t\times (u\times u_{tx}), u_{tx}\>dx\right|\\
=&\left|\int_I\<u_t, u_{tx}\>\<u,u_{tx}\>dx-\int_I\<u_t,u\>\<u_{tx}, u_{tx}\>dx\right|\\
=&\left|\int_I\<u_t, u_{tx}\>\<u_x,u_{t}\>dx\right|\leq \int_I|u_t|^2|u_x||u_{tx}|dx\\
\leq& C(\norm{u}^2_{H^2}+\norm{u_t}^2_{L^2}+1)^2(\norm{u_{tx}}^2_{L^2}+1).
\end{align*}
\begin{align*}
|A_3|=&2\left|\int_I\<u_t\times u, u_{tx}\>\<u_{xx}, u_x\>dx\right|\leq\int_I|u_t||u_{tx}||u_{xx}|u_x|dx\\
\leq &C|u_t|_{L^\infty}|u_x|_{L^\infty}\norm{u_{tx}}_{L^2}\norm{u_{xx}}_{L^2}\leq C\norm{u}^2_{H^2}\norm{u_t}_{H^1}\norm{u_{tx}}_{L^2}\\
\leq &C(\norm{u}^2_{H^2}+\norm{u_t}^2_{L^2}+1)^2(\norm{u_{tx}}^2_{L^2}+1).
\end{align*}
and
\begin{align*}
|A_4|=&\left|\int_I|u_x|^2\<u_t\times u_x, u_{tx}\>dx\right|\leq \int_I|u_x|^3|u_t||u_{tx}dx\\
\leq &C\norm{u}^3_{H^2}(\norm{u_t}^2_{L^2}+\norm{u_{tx}}^2_{L^2})\\
\leq &C(\norm{u}^2_{H^2}+\norm{u_t}^2_{L^2}+1)^3(\norm{u_{tx}}^2_{L^2}+1).
\end{align*}
Therefore, we have
\begin{align*}
\frac{1}{2}\frac{\p}{\p t}\int_{I}|u_{tx}|^2dx\leq C(\norm{u}^2_{H^2}+\norm{u_t}^2_{L^2}+1)^3(\norm{u_{tx}}^2_{L^2}+1),
\end{align*}
which implies
\begin{equation}\label{add}
\sup_{0<t<T}\int_{I}|u_{tx}|^2dx\leq e^{C(\norm{u_0}_{H^1},\, \norm{\tau(u_0)}_{L^2})T}\int_{I}|u_{tx}|^2|_{t=0},
\end{equation}
where
\[u_{tx}(x,0)=u_{0x}\times u_{0xx}+u_0\times u_{0xxx}.\]
Hence, we have
\[\sup_{0<t<T}\int_{I}|u_{tx}|^2\leq C(\norm{u_0}_{H^3},T).\]

Furthermore, we apply the formula \eqref{1eq-H3} to give a bound
\begin{align*}
\int_{I}|u_{xxx}|^2dx\leq &C\{\int_{I}|u_x\times u_t|^2dx+\int_{I}|u\times u_{tx}|^2dx\}\\
&+C\{\int_{I}|u_{xx}\cdot u_x|^2dx+\int_{I}|u_x|^6dx\}\\
=&B_1+B_2+B_3+B_4,
\end{align*}
where
\begin{align*}
|B_1|=& C\int_{I}|u_x\times u_t|^2dx\leq C\norm{u_x}^2_{L^\infty}\norm{u_t}^2_{L^2}\leq C\norm{u}^2_{H^2}\norm{u_t}^2_{L^2},\\
|B_2|=&C\int_{I}|u\times u_{tx}|^2dx\leq C\norm{u_{tx}}^2_{L^2},\\
|B_3|=&C\int_{I}|u_{xx}\cdot u_x|^2dx\leq C|u_x|^2_{L^\infty}\norm{u_{xx}}^2_{L^2}\leq C\norm{u}^4_{H^2},\\
|B_4|=&C\int_{I}|u_x|^6dx\leq \norm{u}^6_{H^2}.
\end{align*}
We concludes from the above
\begin{equation}
	\int_{I}|u_{xxx}|^2dx\leq C\norm{u_{tx}}^2_{L^2}+(\norm{u}^2_{H^2}+1)^3\leq C(\norm{u_0}_{H^3}, T).
\end{equation}

Therefore, we get the following $H^3$-bound of $u$.
\begin{prop}
For any $T<T_{max}$, the solution $u$ satisfies an energy estimate:
\[\sup_{0<t<T}\(\norm{u}^2_{H^3}+\norm{u_t}^2_{H^1}\)\leq C(\norm{u_0}_{H^3}, T).\]
\end{prop}
\subsection{Higher order energy estimates}
Now we return to equation \eqref{eq-H3}, i.e.
\[u_{tt}=u_t\times u_{xx}+u\times u_{txx},\]
by differentiating the above equation with respect to $t$ we can show
\begin{equation}\label{eq-H5}
u_{ttt}=u_{tt}\times u_{xx}+2u_t\times u_{txx}+u\times u_{ttxx}.
\end{equation}
Moreover, differentiating the following with respect to $t$
\[u_{xx}=-u\times u_t-|u_{x}|^2u\]
leads to
\begin{equation}\label{1eq-H5}
u_{txx}=-u\times u_{tt}-2\<u_{tx}, u_x\>u-|u_{x}|^2u_t.
\end{equation}

Then, taking $u_{tt}$ as a test function to \eqref{eq-H5}, we obtain
\begin{equation}\label{add1}
\begin{aligned}
\frac{1}{2}\frac{\p}{\p t}\int_{I}|u_{tt}|^2dx&=2\int_{I}\<u_t\times u_{txx}, u_{tt}\>dx+\int_{I}\<u\times u_{ttxx}, u_{tt}\>dx\\
=&-2\int_{I}\<u_t\times (u\times u_{tt}), u_{tt}\>dx-4\int_{I}\<u_{tx}, u_x\>\<u_t\times u, u_{tt}\>dx\\
&-\int_{I}\<u\times u_{ttx}, u_{ttx}\>dx-\int_{I}\<u_x\times u_{ttx}, u_{tt}\>dx\\
=&-2\int_{I}\<u_t,u_{tt}\>\<u, u_{tt}\>dx-4\int_{I}\<u_{tx}, u_x\>\<u_t\times u, u_{tt}\>dx\\
&-\int_{I}\<u_x\times u_{ttx}, u_{tt}\>dx\\
=&2\int_{I}\<u_t,u_{tt}\>\<u_t, u_{t}\>dx+4\int_{I}\<u_{tx}, u_x\>\<u_t\times u, u_{tt}\>dx\\
&-\int_{I}\<u_x\times u_{ttx}, u_{tt}\>dx\\
\leq &2\norm{u_t}^3_{L^6}\norm{u_{tt}}_{L^2}+4\norm{u_{tx}}_{L^2}\norm{u_x}_{L^\infty}\norm{u_t}_{L^6}\norm{u_{tt}}_{L^3}\\
&+\norm{u_x}_{L^6}\norm{u_{tt}}_{L^3}\norm{u_{ttx}}_{L^2}\\
\leq &C(\norm{u_t}^2_{H^1}+\norm{u}^2_{H^2}+1)^3(\norm{u_{tt}}^2_{H^1}+1).
\end{aligned}
\end{equation}

Next, we choose $-u_{ttxx}$ as a test function to \eqref{eq-H5}. Since $u_{ttx}(0,t)=0$ and $u_{ttx}(1,t)=0$, we can deduce that there holds
\begin{equation}\label{add2}
\begin{aligned}
\frac{1}{2}\frac{\p}{\p t}\int_{I}|u_{ttx}|^2dx=&-\int_{I}\<u_{tt}\times u_{xx},u_{ttxx}\>dx-2\int_{I}\<u_t\times u_{txx},u_{ttxx}\>dx\\
=&\int_{I}\<u_{tt}\times u_{xxx},u_{ttx}\>dx+2\int_{I}\<u_{tx}\times u_{txx},u_{ttx}\>dx\\
&+2\int_{I}\<u_t\times u_{txxx},u_{ttx}\>dx=D_1+D_2+D_3,
\end{aligned}
\end{equation}
where
\begin{align*}
|D_1|\leq &C\norm{u_{tt}}_{L^\infty}\norm{u_{xxx}}_{L^2}\norm{u_{ttx}}_{L^2}\leq C\norm{u}_{H^3}\norm{u_{tt}}^2_{H^1},
\end{align*}
\begin{align*}
|D_2|\leq& C\int_{I}|u_{tx}||u_{txx}||u_{ttx}|dx
\leq C|u_{txx}|_{L^\infty}\norm{u_{tx}}_{L^2}\norm{u_{ttx}}_{L^2}\\
\leq &C\norm{u_{txx}}_{W^{1,1}}\norm{u_{tx}}_{L^2}\norm{u_{ttx}}_{L^2},
\end{align*}
and
\begin{align*}
|D_3|= &2\left|\int_{I}\<u_t\times u_{txxx},u_{ttx}\>dx\right|\leq C|u_t|_{L^\infty}\norm{u_{txxx}}_{L^2}\norm{u_{ttx}}_{L^2}\\
\leq&C\norm{u_t}_{H^1}\norm{u_{txxx}}_{L^2}\norm{u_{ttx}}_{L^2}.
\end{align*}

Now, we turn to estimating $\norm{u_{txx}}_{W^{1,1}}$ and $\norm{u_{txxx}}_{L^2}$ which appeared in the above inequalities.
By differentiating the above equation \eqref{1eq-H5} with respect to $x$ we have
\begin{align*}
u_{txxx}=&-u_x\times u_{tt}-u\times u_{ttx}-2\<u_{tx},u_{xx}\>u-2\<u_{txx},u_{x}\>u\\
&-2\<u_{tx},u_{xx}\>u_x-2\<u_{xx},u_x\>u_t-|u_x|^2u_{xt}.
\end{align*}
So, from the two equations obtained in the above respectively we take a simple computation to show
\begin{equation}\label{add3}
\begin{aligned}
\int_{I}|u_{txx}|^2dx\leq &C\{\int_{I}|u_{tt}|^2dx+\int_{I}|u_{tx}|^2|u_x|^2dx+\int_{I}|u_{x}|^4|u_t|^2dx\}\\
\leq &C(\norm{u_{tt}}^2_{L^2}+|u_x|^2_{L^\infty}\norm{u_{tx}}^2_{L^2}+|u_x|^4_{L^\infty}\norm{u_t}^2_{L^2})\\
\leq &C\norm{u_{tt}}^2_{L^2}+C(\norm{u}^2_{H^2}+1)^2\norm{u_t}^2_{H^1},
\end{aligned}
\end{equation}
and
\begin{equation}\label{add4}
\begin{aligned}
\int_{I}|u_{txxx}|^2dx\leq &C\{\int_{I}|u_x\times u_{tt}|^2dx+\int_{I}| u_{ttx}|^2dx\}\\
&+C\{\int_{I}|\<u_{tx},u_{xx}\>|^2dx+\int_{I}|\<u_{txx},u_{x}\>|^2dx\}\\
&+C\{\int_{I}|\<u_{tx},u_{xx}\>u_x|^2dx+\int_{I}|\<u_{xx},u_{x}\>u_t|^2dx+\int_{I}|u_x|^4|u_{xt}|^2dx\}\\
\leq &C\{|u_x|^2_{L^\infty}\norm{u_{tt}}^2_{L^2}+\norm{u_{ttx}}^2_{L^2}+|u_{xx}|^2_{L^\infty}\norm{u_{tx}}^2_{L^2}+|u_{x}|^2_{L^\infty}\norm{u_{txx}}^2_{L^2}\}\\
&+C\{|u_x|^4_{L^\infty}\norm{u_{tx}}^2_{L^2}+\norm{u_x}^2_{L^\infty}\norm{u_{xx}}^2_{L^4}\norm{u_t}^2_{L^4}+|u_x|^2_{L^\infty}|u_{xx}|^2_{L^\infty}\norm{u_t}^2_{L^2}\}\\
\leq &C(\norm{u}^2_{H^3}+\norm{u_t}^2_{H^1}+1)^4(\norm{u_{tt}}^2_{H^1}+1).
\end{aligned}
\end{equation}
Thus, we have obtained the estimates of $\norm{u_{txx}}_{H^1}$ and $\norm{u_{txxx}}_{L^2}$. Moreover, in view of \eqref{add3}, \eqref{add4} and the fact
$$\norm{u_{txx}}_{W^{1,1}}\leq \norm{u_{txx}}_{H^1},$$
we take a simple calculation to show
\[|D_1|+|D_2|+|D_3|\leq C(\norm{u}^2_{H^3}+\norm{u_t}^2_{H^1}+1)^4(\norm{u_{tt}}^2_{H^1}+1).\]
Therefore, we can derive from \eqref{add2}
\begin{equation}\label{add2'}
\frac{1}{2}\frac{\p}{\p t}\int_{I}|u_{ttx}|^2dx\leq C(\norm{u}^2_{H^3}+\norm{u_t}^2_{H^1}+1)^4(\norm{u_{tt}}^2_{H^1}+1).
\end{equation}

Now, we combine \eqref{add1} and \eqref{add2'} to obtain
\begin{align*}
	\frac{1}{2}\frac{\p}{\p t}\norm{u_{tt}}^2_{H^1}\leq C(\norm{u}^2_{H^3}+\norm{u_t}^2_{H^1}+1)^4(\norm{u_{tt}}^2_{H^1}+1),
\end{align*}
which implies that there holds true
\begin{align*}
	\sup_{0<t<T}\norm{u_{tt}}^2_{H^1}\leq C(\norm{u_0}_{H^5}, T).
\end{align*}
Once the above estimate is established, it follows from\eqref{add}, \eqref{add3} and \eqref{add4} that
\begin{align*}
	\sup_{0<t<T}\norm{u_{t}}^2_{H^3}\leq C(\norm{u_0}_{H^5}, T).
\end{align*}

On the other hand, since $u_{xx}=-u\times u_t-|u_{x}|^2u$ and $$\sup_{0<t<T}\norm{u}^2_{H^3}\leq C(\norm{u_0}_{H^3},T),$$
it is not difficult to show
\[\sup_{0<t<T}\norm{u_{xx}}^2_{H^2}\leq C(\norm{u_0}_{H^5},T),\]
which then implies
\[\sup_{0<t<T}\norm{u}^2_{H^4}\leq C(\norm{u_0}_{H^5},T).\]

Furthermore, applying the equation again
\[u_{xx}=-u\times u_t-|u_{x}|^2u,\]
we can show
\[\sup_{0<t<T}\norm{u_{xx}}^2_{H^3}\leq C(\norm{u_0}_{H^5},T),\]
namely,
\[\sup_{0<t<T}\norm{u}^2_{H^5}\leq C(\norm{u_0}_{H^5},T).\]

In other word, we get the following estimate of $u$.
\begin{prop}\label{H5-es}
The solution $u$ has the following energy bound
\[\sup_{0<t<T}\(\norm{u}^2_{H^5}+\norm{u_t}^2_{H^3}+\norm{u_{tt}}^2_{H^1}\)\leq C(\norm{u_0}_{H^5}, T).\]
\end{prop}

With the above $H^5$-energy estimates of $u$ in Proposition \ref{H5-es} at hand, we then apply almost the same argument as that in Section \ref{s: hig-reg} with $\ep=0$ to show the higher order energy bounds of $u$ as follows.
\begin{prop}\label{es-hig-u}
Let $k\geq 2$. For any $T<T_{max}$, there exists constant $C_k(\norm{u_0}^2_{H^{2k+1}}, T)$ such that for any $0\leq i\leq k$ the solution $u$ satisfies the following energy bound
\begin{equation}\label{es-u}
\sup_{0<t<T}\norm{\p^i_tu}^2_{H^{2(k-i)+1}}\leq C_k(\norm{u_0}^2_{H^{2k+1}}, T).
\end{equation}
\end{prop}

\medskip
\subsection{Global existence result}
Now, we are in position to prove the main result (namely Theorem \ref{thm3}) in this section.
\begin{thm}\label{thm3'}
Suppose that $u_0\in C^\infty (I,\U^2)$, which satisfies the $k$-order compatibility condition defined in \eqref{com-cond3} for any $k\in \mathbb{N}$. Then the initial-Neumann boundary value problem \eqref{eq-ISMF} admits a unique global smooth solution $u$ on $[0,\infty)$.
\end{thm}
\begin{proof}
Let $T_{max}$ be the maximal existence time given in Theorem \ref{thm4}. We claim $T_{max}=\infty$.  On the contrary, if $T_{max}<\infty$, then Proposition \ref{es-hig-u} implies that for any $k\in \mathbb{N}$, we have
\[\sup_{0<t<T_{max}}\norm{\p^i_tu}^2_{H^{2k+1-2i}}\leq C_k(\norm{u_0}_{H^{2k+1}}, T_{max})<\infty\]
for $0\leq i\leq k$. This yields that
\[\sup_{0<t<T_{max}}\norm{\p^j_t\p^s_xu}^2_{L^2}<\infty\]
for any $j,s\in \mathbb{N}$.

So, it follows from the Sobolev embedding theorem that
\[u\in C^\infty(I\times[0,T_{max}]),\]
and hence $u(x,T_{max})$ satisfies the same compatibility condition as that for $u_0$ defined in \eqref{com-cond3}. Therefore, $T_{max}$ can be extended. This leads to a contradiction with the definition of $T_{max}$.
\end{proof}

\medskip
\noindent {\it\bf{Acknowledgements}}: The authors thank the referees for their insightful comments, which greatly improved the manuscript. The author B. Chen is supported partially by NSFC (Grant No. 12301074), Guangzhou Basic and Applied Basic Research Foundation (Grant No. 2024A04J3637) and Guangdong Basic and Applied Basic Research Foundation (Grant no. 2025A1515010502), the author Y.D. Wang is supported partially by NSFC (Grant No.12431003).


\end{document}